\documentclass[10pt]{amsart}
\usepackage[left=1in,right=1in,top=1in,bottom=1in]{geometry}
\usepackage[utf8]{inputenc}
\usepackage{setspace}
\usepackage{amsfonts}
\usepackage{amsthm}
\usepackage{thmtools}
\usepackage{amsmath}
\usepackage{mathrsfs}
\usepackage{mathtools}
\usepackage{amssymb}
\usepackage{bbm}
\usepackage[T1]{fontenc}
\usepackage{esint}
\usepackage{xcolor}
\usepackage{hyperref}
\usepackage[noabbrev,capitalize,nameinlink]{cleveref}

\newcommand{\R}{\mathbb{R}}
\newcommand{\Pbb}{\mathbb{P}}
\newcommand{\E}{\mathbb{E}}

\newcommand{\Prob}{\mathrm{Prob}}

\newcommand{\C}{\mathbb{C}}
\newcommand{\N}{\mathbb{N}}

\newcommand{\pair}[2]{\left\langle #1, #2 \right\rangle}

\newcommand{\id}{\mathrm{id}}

\newcommand{\ind}{\mathbbm{1}}

\newcommand{\gr}{\mathrm{gr}}
\newcommand{\Perm}{\mathrm{Perm}}
\newcommand{\End}{\mathrm{End}}
\newcommand{\Aut}{\mathrm{Aut}}

\newtheorem{theorem}{Theorem}[section]

\newtheorem{remark}{Remark}[section]
\newtheorem{prop}[theorem]{Proposition}

\newtheorem{fact}{Fact}
\newtheorem{cor}[theorem]{Corollary}

\theoremstyle{definition}
\newtheorem{definition}{Definition}[section]

\crefname{theorem}{Theorem}{Theorems}
\crefname{lemma}{Lemma}{Lemmas}
\crefname{remark}{Remark}{Remarks}
\crefname{prop}{Proposition}{Propositions}
\crefname{conjecture}{Conjecture}{Conjectures}
\crefname{fact}{Fact}{Facts}
\crefname{cor}{Corollary}{Corollaries}

\title{On the Composition of Permutons}
\author{Christian Gorski}
\address{University of Washington}
\email{cgorski1@uw.edu}

\date{October 2026}

\begin{document}
\begin{abstract}
    %We construct a natural associative binary operation on permutons analogous to composition of permutations
    %and prove both positive and negative continuity results.
    %In particular, composition is not jointly continuous in the usual topology on permutons (the weak topology),
    %but we introduce another topology, which we call the strong topology, with respect to which it is continuous.
    %We also show that a sequence of permutons coming from a sequence of permutations
    %on a set of diverging size converges in the strong topology if and only if
    %it converges in the weak topology, and the transpose of the limit permuton
    %is the pushforward of Lebesgue measure on $[0,1]$ along the map $x \mapsto (x, F(x))$
    %for some Lebesgue-measure-preserving map $F: [0,1] \to [0,1]$.
    Permutons can be thought of as ``continuum limits'' of permutations of diverging size, which 
    encode information about the number of occurrences of any given finite subpattern.
    We construct a natural composition operation on permutons, analogous
    to composition of permutations, and explore its properties via a correspondence
    between permutons, Markov kernels, and Markov operators.
    
    We show that no extension of permutation composition to permutons can be (jointly) continuous with
    respect to the usual ``weak'' topology on permutons,
    but we construct a ``strong'' topology with respect to which composition is jointly continuous.
    Moreover, we characterize strong convergence of sequences of permutations of diverging size:
    $\sigma_N \to P$ strongly if and only if $\sigma_N \to P$ weakly and $P$
    is the transpose of the graph of a measure-preserving function. The key insight behind these results is to apply functional analysis to certain operators associated to permutons; with this shift in perspective, our results have elegant and relatively ``soft'' proofs.

    As an application, we describe the permuton limit of a product of finitely many independent uniformly sampled separable 
    permutations and give a formula for the expectation of this limit.
    We also show that in this case, as the number of compositions goes to infinity, 
    the aforementioned expectation tends
    to the uniform permuton.
    \end{abstract}

\maketitle

\section{Introduction}
A \emph{permuton} is a probability measure on $[0,1]^2$ with uniform marginals; we will denote
the set of all permutons by $\Perm$.
We can encode any \emph{permutation} $\sigma \in S_N$
as a permuton $P_\sigma \in \Perm$ by taking the following uniform measure:\footnote{As usual if $S \subseteq \R^d$,
then we define the uniform measure $\mathrm{Unif}(S)$
by $\mathrm{Unif}(S)(A) = \mathcal{L}^d(A \cap S) / \mathcal{L}^d(S)$,
where $\mathcal{L}^d$
is Lebesgue measure on $\R^d$.}
\begin{equation} \label{eq:permutation to permuton}
    P_\sigma := \mathrm{Unif}\left(\bigcup_{i=1}^N \left[\frac{i-1}{N},\frac{i}{N}\right] \times \left[\frac{\sigma(i)-1}{N},\frac{\sigma(i)}{N}\right] \right ).
\end{equation}
Note the close relationship to the graph of the function
$\sigma:\{1,...,N\} \to \{1,...,N\}$.
We say that a sequence of permutons $P_N$
converges to a permuton $P$ if for each
continuous $f \in C([0,1]^2)$, we have $\int f dP_N \to \int f dP$, i.e. $P_N \to P$
in the weak-* topology on $\Prob([0,1]^2) \subset C([0,1]^2)^*$.
We say that a sequence of permutations $\sigma_N \in S_N$
(with $N \to \infty$) converges to a permuton $P$
if the associated permutons $P_{\sigma_N}$ converge to $P$.
With this topology, $\Perm$ is a compact (separable metrizable) topological space, and $\bigcup_{n=1}^\infty \{ P_\sigma : \sigma \in S_n\}$ is a dense subset.\footnote{The fact that any permuton is a limit of permutations is shown in Lemma 4.2 in Hoppen et al. \cite{Hoppen2013}.}
Therefore, $\Perm$ can be thought of as a \emph{compactification} of the set of permutations,
and a permuton $P$ can be thought of as a ``continuum limit'' of a sequence of permutations on a set of diverging size.

Permutons as limits of permutations
were first introduced by Presutti and Stromquist \cite{PresuttiStromquist2010}, 
and the theory was later developed by Hoppen, Kohayakawa, Moreira, R\'{a}th, and Sampaio in \cite{Hoppen2013}.
The relevance of permutons
to combinatorics comes from the following fact,
proved in \cite{Hoppen2013}: a sequence $\sigma_N \in S_N$ of permutations with $N \to \infty$ converges to a permuton %$P$
if and only if, for each fixed $\tau \in S_k$,
the \emph{subpattern density} of $\tau$ in $\sigma_N$
converges;
moreover, the limiting subpattern densities can be recovered from
the limit permuton $P$.
%to $\widetilde{\mathrm{occ}}(\tau,P)$,
%a quantity which can be expressed as a certain integral against $P^{\otimes k}$.
For precise statements and a more thorough discussion, we refer the reader to Section 2.1 of \cite{BorgaThesis}; but the point we emphasize here is that permuton limits encode a certain type of combinatorial information quite precisely. This has made permutons useful for studying pattern-avoiding permutations (e.g. \cite{MaazounThesis, Bassino2016, Bassino2020, Garbe2024}), as well as random permutations on sets of diverging size more generally (e.g. \cite{Runsort, Starr2009, Jin2019, Kenyon2020, BorgaThesis}).
Recent work has also connected permutons to fundamental random geometric models in two dimensions \cite{BorgaBaxterLiouville, Borga2024Reconstructing, BorgaGwynneSun2026}.

Since permutations have a natural composition operation coming from composition of functions, it is natural to ask
whether there is a corresponding operation on permutons.
In this article, we construct just such an operation and explore its properties; we also apply this understanding to describe the permuton limits of compositions of independent uniformly sampled separable permutations.
A key insight behind our proofs is to use a correspondence between permutons and certain linear operators, and then apply methods from functional analysis.

Our main theorem is a continuity result 
concerning our composition operation (which we will construct in \cref{sec:markov kernels} below):
\begin{theorem} \label{main}
    Let $(\sigma_N)_{N \in \N}$ and $(\tau_N)_{N \in \N}$ be two sequences of permutations,
    with each $\sigma_N, \tau_N \in S_{k_N}$ with $k_N \to \infty$.
    Suppose that $\sigma_N$ converges to the permuton $P$ 
    and $\tau_N$ converges to the permuton $Q$.
    If there exists a Lebesgue-measure preserving function $F:[0,1] \to [0,1]$ such that either:
    \begin{enumerate}
        \item $P$ is the pushforward of Lebesgue measure along the map $[0,1] \to [0,1]^2, x \mapsto (x,F(x))$, or
        \item $Q$ is the pushforward of Lebesgue measure along the map $[0,1] \to [0,1]^2, y \mapsto (F(y), y)$,
    \end{enumerate}
    then the sequence of 
    compositions $\sigma_N \tau_N$
    converges to $PQ$
    as $N \to \infty$.
\end{theorem}

Several important (random) permuton limits satisfy the hypotheses of \cref{main}, most notably the \emph{(biased) separable Brownian permutons},
which were recently shown to form important universality classes for random permutation models in \cite{Bassino2020}.
Still, the theorem may at first glance appear overly restrictive.
The ideal situation would be to define a jointly continuous composition operation on permutons,
so that $\sigma_N \to P, \tau_N \to Q$ always implies $\sigma_N \tau_N \to PQ$;
then, the permuton limits of the individual permutation sequences would immediately
characterize the permuton limit of the composition.
Unfortunately, the composition operation we exposit here is \emph{not} jointly continuous,
and indeed no jointly continuous extension of composition 
from $\bigcup_{n=1}^\infty \{P_\sigma : \sigma \in S_n\}$ to $\Perm$ exists.

To see this, consider the following counterexample. Let $U$ be the ``uniform permuton,''
which is simply equal to Lebesgue measure on $[0,1]^2$.
To make our argument, we need the following well-known fact,
which can be seen as a special case of Lemma 4.2 in \cite{Hoppen2013}:

\begin{fact}
    For each $n$, sample $\sigma_n$ uniformly from $S_n$. Then
    \[
        \Pbb( \lim_{n \to \infty} \sigma_n = U ) = 1.
    \]
\end{fact}

Now, sample $\sigma_n$ from $S_n$ uniformly, and note that its inverse $\sigma_n^{-1}$ is also uniform on $S_n$.
Therefore we have $\sigma_n, \sigma_n^{-1} \to U$ with probability 1
but $\sigma_n \sigma_n^{-1} = \id_n \to \Delta$,
where $\Delta$ is the ``diagonal permuton,'' which is the pushforward
of Lebesgue measure along the diagonal map $[0,1] \to [0,1]^2, x \mapsto (x,x)$.
Thus, any jointly continuous composition operation must satisfy $UU = \Delta$.

On the other hand, if we sample $\sigma_n, \tau_n$ \emph{independently} uniformly from $S_n$,
then $\sigma_n \tau_n$ is also uniform on $S_n$, meaning that 
$\sigma_n, \tau_n, \sigma_n \tau_n \to U$; therefore any 
jointly continuous composition operation must also satisfy $UU = U$.
Hence no such operation can exist.

In spite of this, the composition we present enjoys (at least) three good properties:
\begin{enumerate}
    \item Composition is continuous with one variable fixed, that is, if $P_n \to P \in \Perm$,
        then for all $Q \in \Perm$ we have $P_n Q \to PQ$ and $QP_n \to QP$ (see \cref{continuity fixing one component}).
    \item There is another natural topology on $\Perm$, the \emph{strong} topology,
        with respect to which composition \emph{is} jointly continuous.
        Moreover, we provide a neat characterization of strong convergence
        for sequences of permutations of diverging size (see \cref{characterization of strong convergence}).
    \item Composition is compatible with taking expectations of random permutons (see \cref{expectation composition}).
        That is, if $P$ and $Q$ are independent \emph{random} permutons, then
        \[ \E[PQ] = \E[P] \E[Q]. \]
\end{enumerate}

Our main theorem will be an immediate consequence of strong continuity of composition 
(which is straightforward to show) together with our characterization of
strong convergence of permutations.
Point (3) will be helpful for understanding and describing some of the limits we obtain from applying \cref{main} a bit more explicitly; see \cref{application} below.
%The third property will then allow
% us to deduce some facts about expected
%subpattern densities of independent products of certain random permutations;
%in particular, we get explicit formulas for the \emph{expectation}
%of the limit of some sequences of compositions of random permutons.
%We also produce a simple sufficient condition for ``mixing'' of a permuton,
%which will tell us that certain large independent products are nearly uniform in expectation.

A straightforward application of \cref{main}
yields the following. Recall that $\Delta$ denotes the diagonal permuton.
\begin{cor} \label{corollary diagonal}
    Suppose that $\sigma_N, \tau_N$
    are (possibly random) permutations
    and $P \in \Perm$ (possibly random) with
    $\sigma_N \to \Delta$ (almost surely) and $\tau_N \to P$ (in distribution).
    Then
    \[
        \sigma_N \tau_N, \tau_N \sigma_N \to P
    \]
    (in distribution).
    For instance, if $\sigma_N$
    is a uniformly sampled permutation 
    avoiding\footnote{For a definition of what it means for a permutation to contain or avoid a pattern, we refer the reader to Section 2.1 of \cite{BorgaThesis}.} the pattern $(231)$,
    and $\tau_N \to P$,
    then $\sigma_N \tau_N, \tau_N \sigma_N \to P$.
    Moreover, if $\sigma_N^{(1)},...,\sigma_N^{(k)}$
    are uniformly sampled $(231)$-avoiding permutations, we have
    \[
       \sigma_N^{(1)} \cdots \sigma_N^{(k)} \to \Delta
    \]
    as $N \to \infty$, and for each $i=1,...,k$ and $1 \le \ell < \infty$ we have the convergence of the $\ell$\textsuperscript{th} power
    \[
        (\sigma_N^{(i)})^{\ell} \to \Delta
    \]
    almost surely as $N \to \infty$.
\end{cor}
Note that the above requires no extra assumptions on $P$, nor even any \emph{independence} assumptions on the sequences of permutations; if a random permutation converges to $\Delta$,
composing it with another random permutation will not affect the latter's limit, even if the permutations are strongly correlated. The fact that arbitrary finite powers and independent products of $231$-avoiding permutations also approach $\Delta$ is not a priori obvious, as the property of being $231$-avoiding is not deterministically preserved by compositions or powers.

As a more involved application of \cref{main}, we also prove the following fact about compositions of random separable permutations\footnote{A \emph{separable} permutation is one which avoids the subpatterns 2413 and 1342; see e.g. \cite{Bassino2016}.}:

\begin{theorem} \label{application}
    For each $N \in \N$, let
    $\sigma_N^{(1)},\sigma_N^{(2)},...$ be an i.i.d. sequence
    of random permutations
    sampled uniformly from the 
    set of all separable permutations in $S_N$.
    For each fixed $k$, as $N \to \infty$, we have
    \[
        \sigma_N^{(1)} \cdots \sigma_N^{(k)} \to P_1 \cdots P_k,
    \]
    in distribution,
    where $P_1,P_2,...$
    is an i.i.d. sequence
    of \emph{separable Brownian permutons} as defined in \cite{Bassino2016}.
    Moreover, the expectation of this limit satisfies
    \[
        \E[P_1 \cdots P_k] = \E[P_1]^k,
    \]
    and in particular, is absolutely continuous with respect to Lebesgue measure, with density $\gamma \in L^1([0,1]^2)$
    given by
    \[
        \gamma(x,y)
        =
        \int_{0}^1 \cdots \int_0^1
        \alpha(z_k,y) \left(\prod_{i=1}^{k-1}
        \alpha(z_i,z_{i+1})\right)
        \alpha(x,z_1) dz_1 \cdots dz_k,
    \]
    where $\alpha = \alpha_{1/2}$
    is as defined in Theorem 1.7 of \cite{Maazoun2017}.
    Moreover, the expectation of this limit \emph{mixes} as the number of compositions goes to infinity, that is:
    \[
        \lim_{k \to \infty} \E[P_1 \cdots P_k] = \lim_{k \to \infty} \E[P_1]^k = U
    \]
    where $U = \mathcal{L} \otimes \mathcal{L}$ is the uniform permuton.
\end{theorem}

\begin{figure}[h]
\includegraphics[width=0.3\textwidth]{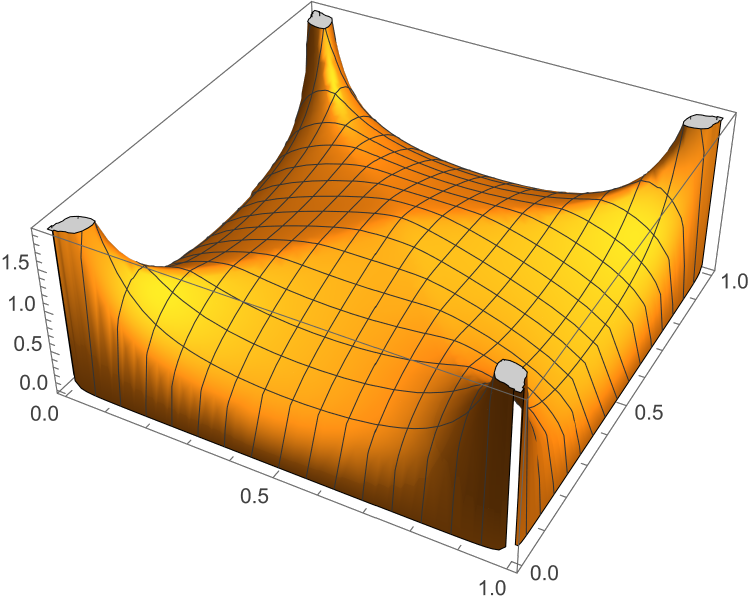}
\includegraphics[width=0.3\textwidth]{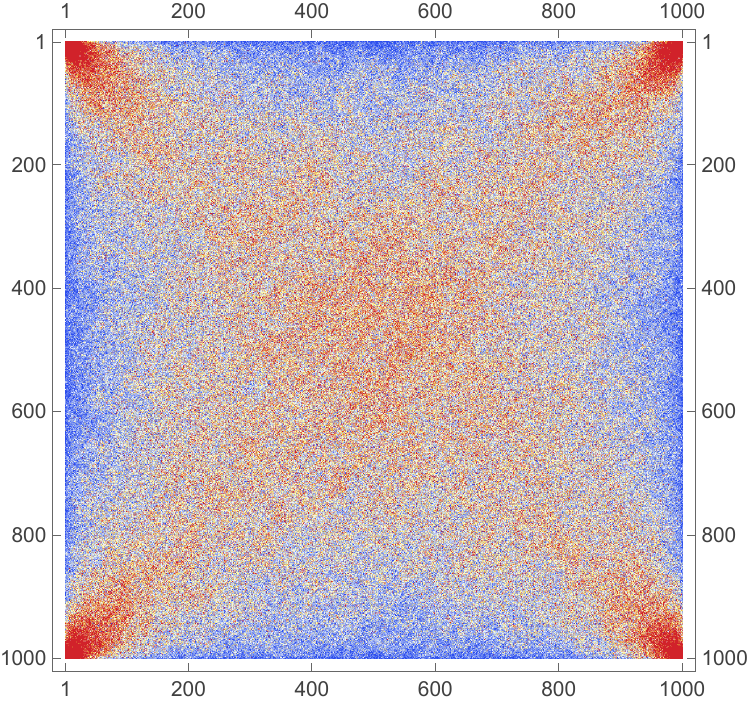}
\caption{Left: Graph of the density $\gamma$ of $\E[P_1 P_2]$,
where $P_1,P_2$ are independent separable Brownian permutons.
Right: Heatmap of 2500 independent samples of $\sigma \tau$,
where $\sigma, \tau$ are independent uniform separable permutations from $S_{1000}$.
Software for sampling uniform separable permutations by Benedikt Stufler. Visualizations produced by Mathematica.}
\label{fig:graph-vs-heatmap}
\end{figure}

\begin{figure}[h]
\includegraphics[width=0.3\textwidth]{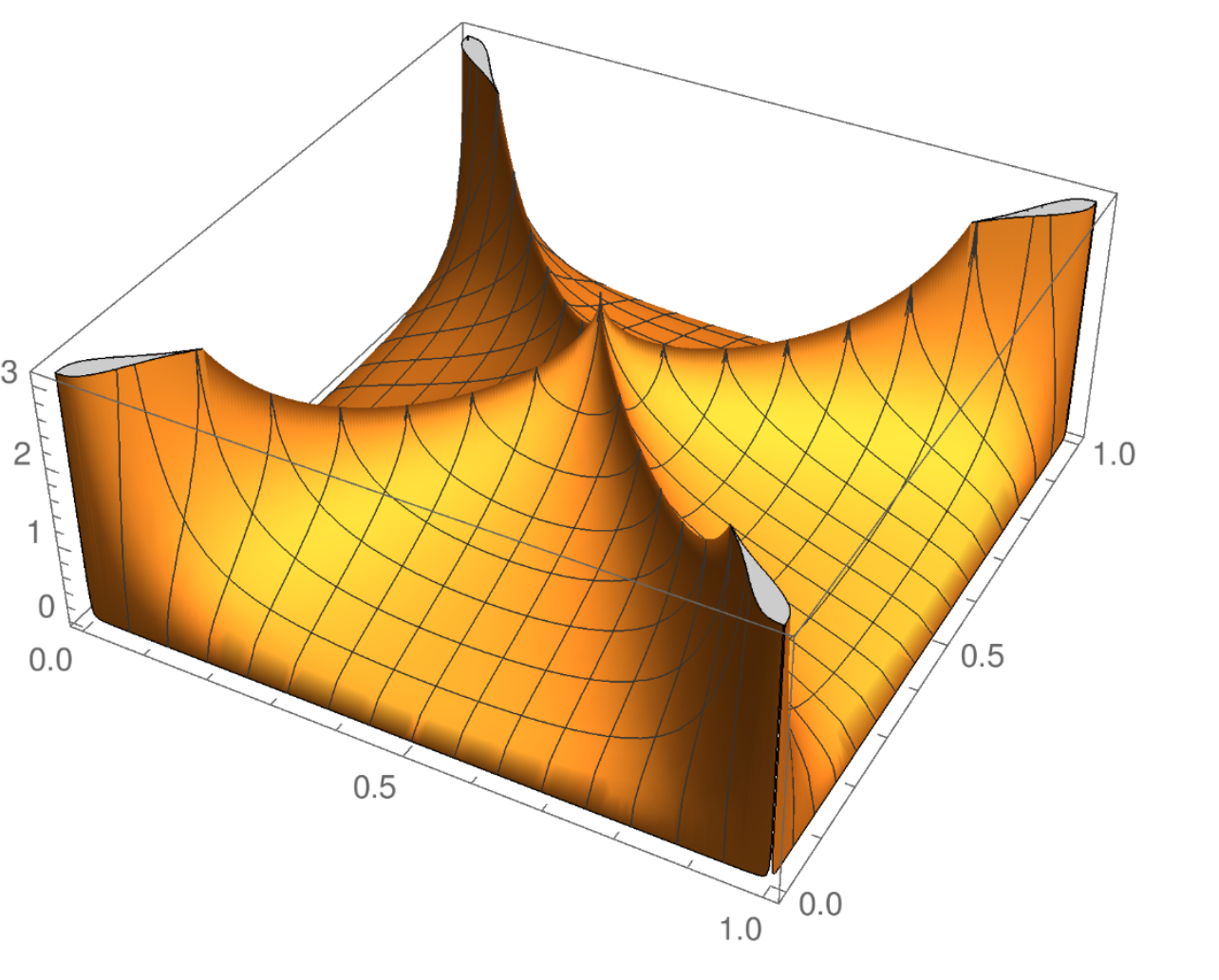}
\includegraphics[width=0.3\textwidth]{composition-of-two-separable-graph.pdf}
\includegraphics[width=0.3\textwidth]{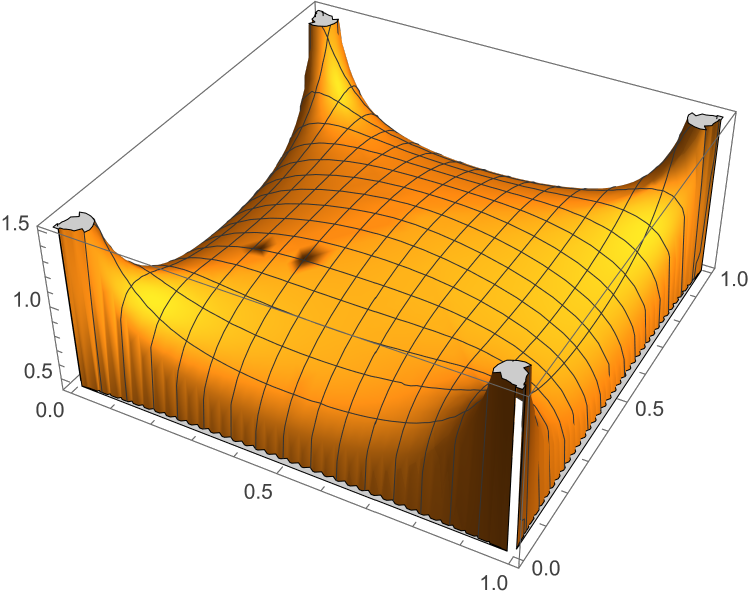}
\caption{Graphs of densities of $\E[P_1], \E[P_1 P_2],$ and $\E[P_1 P_2 P_3]$
for $P_1, P_2, P_3$ independent Brownian separable permutons.
We see that as the number of compositions increases, the density becomes smoother and closer to uniform, consistent
with the final statement in \cref{application}.}
\label{fig:graphs-of-compositions}
\end{figure}

See \cref{fig:graph-vs-heatmap} and  \cref{fig:graphs-of-compositions}
for some visualizations of the density $\gamma$ for $k=1,2,3$.
To show this last statement describing
the limit for a large number of \emph{compositions},
we will prove a criterion for ``mixing'' of permutons, see \cref{sec:mixing}. This criterion will follow from the spectral theorem together with a characterization of idempotent permutons.

We define our composition on $\Perm$ and prove facts about it by using two perspectives on $\Perm$. First, we can view permutons as \emph{Markov kernels};
second, we can view them as \emph{Markov operators}.
The latter perspective is particularly useful, as it gives us access to tools from functional analysis. The former perspective we find more intuitive for certain types of arguments (for instance, those involving convexity).

The outline of the paper is as follows.
In \cref{sec:notation} we establish basic notation and definitions. In \cref{sec:markov kernels}
we describe the correspondence between permutons and stationary Markov kernels, use this to construct our composition on $\Perm$, and collect a few basic facts about this composition.
In \cref{sec:markov operators}, we describe the correspondence between permutons and certain linear operators; in particular, composition of permutons corresponds to composition of operators, and the usual topology on permutons corresponds to the weak operator topology.
In \cref{sec:strong convergence}, we use the linear operator perspective to define strong convergence of permutons, and we provide a neat characterization
of strong convergence of a sequence of permutations to a permuton.
In \cref{sec:proof of main}, we combine our characterization of strong convergence of permutations with an elementary functional analysis argument to prove \cref{main} and \cref{corollary diagonal}.
After this, we give the extra ingredients needed to obtain \cref{application}: facts about expectation of permutons and a sufficient condition for ``mixing'' of a permuton (we say a permuton ``mixes'' if its powers tend to the uniform permuton $U = \mathcal{L} \otimes \mathcal{L}$).
In \cref{sec:expectation} we provide a proof that $\E[PQ] = \E[P]\E[Q]$ for independent random permutons $P,Q$. We use this and \cref{main} 
in \cref{sec:first application}
to prove the first part of \cref{application},
which describes the mean limit of the composition of finitely many independent uniformly sampled separable permutations. In \cref{sec:mixing},
we prove our sufficient condition for mixing of a permuton. The two main ingredients here are the Spectral Theorem and a characterization of idempotent permutons.
We use this to complete the proof of \cref{application} in \cref{sec:second application}. As a bonus, we also show that the family of biased separable Brownian permutons is \emph{not} closed under independent composition.
\cref{sec:algebra} compiles the algebraic facts about $\Perm$ that we have shown throughout the paper into one proposition for the interested reader.
\cref{sec:future} discusses open questions and future directions.
The appendices contain proofs of various facts which are not included in the main body of the paper.

\subsection{Previous work}

For an overview of permutons and their application to combinatorics, the reader is referred to \cite{BorgaThesis}.
We have mentioned above several works
\cite{PresuttiStromquist2010, Hoppen2013, BorgaThesis, MaazounThesis, Bassino2016, Bassino2020, Garbe2024,
Runsort, Starr2009, Jin2019, Kenyon2020,
BorgaBaxterLiouville, Borga2024Reconstructing, BorgaGwynneSun2026} which use this technology to understand random permutations and random geometry.
We note that of the works listed, only \cite{Jin2019} addresses the question of \emph{compositions} of random permutations.
\cite{Jin2019} does allude to our notion of composition in a special case,
and anecdotally, many experts will come up with the correct
notion of composition if prompted.
However, the present article is to our knowledge the first attempt
in the literature to define the composition in general
in the context of permutons,
and the first application of a functional analytic
perspective to the question of composition of permutons.

It should be remarked that outside of the context of permutons, the composition operation we construct is known and occurs in at least a few different contexts.
Since a permuton is equivalent to a self-coupling of Lebesgue measure $\mathcal{L}$ on $[0,1]$,
our composition will be a specific instance of composition in the \emph{category of couplings},
as studied in categorical probability \cite{Law, EnsarguetPerrone}.
Moreover, in ergodic theory, the composition constructed
is well-known as a \emph{relatively independent joining},
and the connection to Markov operators is also known (see e.g. Chapter 6.2 of \cite{Glasner2003}).
This operation was also introduced
by Darsow, Nguyen, and Olsen \cite{Copula92}
in the literature in the context of \emph{copulas},
which are mathematically equivalent objects
with a statistical interpretation.
Some work of A. M. Vershik (e.g. \cite{Vershik83})
studies these objects under the name of \emph{polymorphisms}.
Some results on Markov operators related to \cref{characterization of strong convergence} were proved by James R. Brown 
in 
\cite{Brown66}.\footnote{Note that in \cite{Brown66},
the Markov operator associated to a bistochastic measure (i.e. permuton)
is the \emph{transpose} of 
the operator we choose here.
We choose this so that the map from
permutons to operators is
a \emph{homomorphism}
rather than an \emph{antihomomorphism},
see \cref{basic operator facts}.}
From a formal perspective,
most of the intermediate results here 
involving the algebra and topology of $\Perm$ are not new, and have been rediscovered several times in different contexts. The chief novelty therefore comes from the application of these ideas to the study of limits of permutations.

Although most of the components of the proof
of \cref{main} can be found somewhere in the literature
for copulas or Markov operators, we generally provide
proofs here either in the body of the paper or in the appendix, for three reasons. First,
this keeps the present article self-contained and emphasizes the simple and elegant nature of most of the proof ingredients.
Second, it should give the combinatorialist
reading this some indication of 
some methods for arguing about permutons
and their compositions, particularly
in terms of their disintegrations
or associated operators.
Third, 
the proofs we give were developed without knowledge of most of the aforementioned work, and
so tend to take a slightly different perspective
which is not as easy to find in the literature.
For instance, our proof of
\cref{characterization of isometries}
takes advantage of the interplay between
the two descriptions of permutons
as Markov kernels and Markov operators;
in our opinion, it provides
a clearer explanation of 
why isometries
correspond to (transposes of)
graphs of functions
than one could obtain
by using the operator perspective alone.
The choice of which proofs are in the body of the paper and which are deferred to the appendix is largely determined by considerations of readability.

\subsection{Acknowledgments}
The author thanks Valentin F\'{e}ray for providing references, helpful conversations, and suggestions on early versions of the manuscript.
The author also thanks Jacopo Borga, whose mini-course on permutons at the Thematic Program on Randomness and Geometry at the Fields Institute in Winter of 2024
ultimately inspired this project.

\section{Notation and basic definitions} \label{sec:notation}

To save space, we will denote by $I$ the unit interval $[0,1] \subset \R$.
We will denote by $\mathcal{L}$ Lebesgue measure on $I$ and we will denote by $\mathcal{L}^2$ Lebesgue measure on the unit square $I^2$.
We shall also refer to $\mathcal{L}$ and $\mathcal{L}^2$
as the \emph{uniform} measures on $I$ and $I^2$ respectively.
When we integrate with respect to Lebesgue measure, for brevity, we will usually write $dx$ for the volume element rather than $d\mathcal{L}(x)$, that is:
\[
    \int_{x \in I} f(x) dx 
    :=
    \int_{x \in I} f(x) d\mathcal{L}(x).
\]

We define the standard coordinate projections:
\[
\pi_1:I \times I \to I, (x,y) \mapsto x,
\]
\[
\pi_2:I \times I \to I, (x,y) \mapsto y.
\]

For a measurable map $f:X \to Y$ and a measure $\mu$ on $X$, the \emph{pushforward of $\mu$ along $f$} is the measure $f_*\mu$ on $Y$ defined by
\[
    f_* \mu(B) := \mu(f^{-1}(B)) = \mu(\{x \in X : f(x) \in B\}).
\]
We denote the set of all probability measures on a measurable space $X$ by $\Prob(X)$.\footnote{Here,
$X$ is equipped with some fixed $\sigma$-algebra $\mathcal{F}$ which we suppress in our notation;
typically $X$ will be a topological space,
and the underlying $\sigma$-algebra will be the Borel $\sigma$-algebra, unless otherwise specified.}

\begin{definition}
    A \emph{permuton} 
    is a probability measure
    $P \in \Prob(I \times I)$
    on $I \times I$
    whose marginals are uniform; that is, such that
    \[
    (\pi_1)_* P =
    \mathcal{L} =
    (\pi_2)_* P.
    \]
\end{definition}

As further notation:
\begin{itemize}
\item 
We denote the collection of all permutons by
$\Perm$.
\item 
For a permutation $\sigma \in S_N$,
we define the associated permuton $P_\sigma$
as in \eqref{eq:permutation to permuton}.
\item
We denote by $U$
the \emph{uniform permuton}
$U = \mathcal{L}^2 = \mathcal{L} \otimes \mathcal{L}$.
\item
If $F:I \to I$ is a measure preserving function, then
the \emph{graph of $F$}
is the permuton
\[  
    \gr(F):= 
    (\id \times F)_* \mathcal{L}
\]
where $\id \times F:I \to I \times I$
is the map $(\id \times F)(x) = (x, F(x))$.
\item
We denote by $\Delta$ the \emph{diagonal} or \emph{identity} permuton
\[
    \Delta := \gr(\id)
    = (\id \times \id)_* \mathcal{L}.
\]
\end{itemize}

If $\mu \in \Prob(X)$
and $\nu_{\cdot}:X \to \mathrm{Prob}(Y)$
is a family of
probability measures indexed by $x \in X$,
we define the
integral
$\int \nu_x d\mu(x) \in \Prob(Y)$ by
\[
    \left( \int \nu_x d\mu(x) \right)(B)
    =
    \int \nu_x(B) d\mu(x).
\]
Equivalently, for any bounded measurable $F:Y \to \R$, we have
\[
    \int F(y)
    d\left(\int \nu_x d\mu(x)
    \right)(y)
    =
    \int_{x \in X} \left(\int_{y \in Y} F(y) d\nu_x(y)\right) d\mu(x).
\]

If $\mu \in \Prob(X), \nu \in \Prob(Y)$, we define the \emph{product measure} $\mu \otimes \nu \in \Prob(X \times Y)$
to be the probability measure satisfying
\[
    (\mu \otimes \nu)(A \times B) = \mu(A) \nu(B)
\]
for all measurable $A \subset X, B \subset Y$.

For $x \in X$, we will denote by $\delta_x$ the Dirac measure at $x$, that is,
\[
    \delta_x(B)
    =
    \begin{cases}
        1 & \mbox{if } x \in B \\
        0 & \mbox{otherwise.}
    \end{cases}
\]

Using the above, we have the following equivalent definitions:
\[
    U = 
    \int (\delta_x \otimes \mathcal{L}) dx,
\]
\[
    \gr(F) =
    \int (\delta_x \otimes \delta_{F(x)})dx,
\]
\[
    \Delta =
    \int (\delta_x \otimes \delta_x)dx.
\]

For a set $A$,
we denote by $\ind_A$
the \emph{indicator function}
\[
    \ind_A(x) :=
    \begin{cases}
        1 &\mbox{ if } x \in A \\
        0 &\mbox{ if } x \notin A
    \end{cases}.
\]
More generally, if $\mathcal{P}$
is some proposition,
$\ind(\mathcal{P})$
is equal to $1$
if $\mathcal{P}$
is true, and is equal to $0$ otherwise.

\section{Construction of the composition operation via Markov kernels} \label{sec:markov kernels}

Our understanding of compositions of permutons
will be facilitated by correspondences with two other types of objects: \emph{Markov kernels}
and \emph{Markov operators}.
In this section we give the correspondence
between permutons and Markov kernels, 
define the composition operation,
and discuss a few basic properties
of the composition.

Let us first recall the definition of a Markov kernel.
\begin{definition}
    A \emph{Markov kernel} on $I$ is a map $p:I \to \mathrm{Prob}(I)$, $x \mapsto p_x$ such that for every Borel subset $B \subseteq I$,
    the map $x \mapsto p_x(B)$ is measurable.
\end{definition}
The reader may find it useful to imagine Markov kernels as being ``graphed'' in $I \times I$, with each measure $p_x$ lying on $\{x\} \times I$, sitting ``over'' $x \in I$. It will become even more clear why this picture is appropriate when we discuss disintegrations of measures below.

Markov kernels naturally act on measures:
\begin{definition}
    If $p:I \to \mathrm{Prob}(I)$ is a Markov kernel and $\mu \in \mathrm{Prob}(I)$ is a measure on $I$, then we define
    \[
        p \circ \mu := \int_{x \in I} p_x d\mu(x) \in \mathrm{Prob}(I),
    \]
    where the integral of probability measures above is as defined
    in \cref{sec:notation}.
%    Here, integrals of probability measures are taken in the usual weak sense, so if $f \in C(I)$, we have
%    \[
%        \int f d(p \circ \mu) = \int_{x \in I} \int_{y \in I} f(y) dp_x(y) d\mu(x).
%    \]
\end{definition}

\begin{definition}
    A Markov kernel $p$ is called $\mu$-\emph{stationary} if $p \circ \mu = \mu$.
    Here, if a Markov kernel $p$ on $I$ is $\mathcal{L}$-stationary with $\mathcal{L}$ Lebesgue measure on $I$,
    we may simply say $p$ is \emph{stationary}.
\end{definition}

Now we define composition of Markov kernels:
\begin{definition}
    Let $p$ and $q$ be Markov kernels on $I$. Then we define the composition $p \circ q$ to be the Markov kernel on $I$ given by
    \[
        (p \circ q)_x := p \circ (q_x) = \int p_y d(q_x)(y).
    \]
\end{definition}

It is straightforward to check that composition of Markov kernels is associative, that it is compatible with the action of Markov kernels on measures,
and that a composition of stationary Markov kernels is stationary.

\begin{remark}
Markov kernels naturally arise in the context of Markov processes on a continuous state space.
In this context, if the current state of the system is $x \in I$, then 
the next state of the system is independently sampled from the measure $p_x \in \mathrm{Prob}(I)$.
If we sample the initial state $X_0$ according to a probability measure $\mu \in \Prob(I)$, and then sample $X_1$ according to $p_{X_0}$,
then the distribution of $X_1$ is exactly $(p \circ \mu)$.
Thus, in the context of Markov processes, $p$ being $\mu$-stationary is exactly equivalent to $\mu$
being a stationary measure for the Markov process represented by the kernel $p$.

Moreover, if $p_x$ describes the distribution of the state after one time interval starting from $x \in I$,
then the $n$-fold composition $(p^{\circ n})_x$ describes the distribution of the state after $n$ time intervals starting from $x$.
Generally, the composition $q \circ p$ describes the overall effect of first transitioning according to $p$, then transitioning according to $q$.
We do not strictly need this interpretation 
for our arguments,
but the reader may find it helpful for intuition.
\end{remark}

An important fact we will use is that, for \emph{stationary} Markov kernels, composition is well-defined up to $\mathcal{L}$-null-sets.
\begin{fact} \label{composition is well defined up to ae equivalence}
   Let $p, p', q, q': I \to \mathrm{Prob}(I)$ all be Markov kernels for which Lebesgue measure $\mathcal{L}$ on $I$ is stationary.
   Further suppose that
   \[
      \mathcal{L} \{ x : p_x = p'_x \} = 1 = \mathcal{L} \{ x : q_x = q'_x \}.
   \]
   Then 
   \[
      \mathcal{L} \{ x : (q \circ p)_x = (q' \circ p')_x \} = 1.
   \]
\end{fact}
This fact is well known,
but we include a measure-theoretic proof in the appendix to keep this paper relatively self-contained.
We note that
the stationarity assumption is necessary;
to see this,
consider the (non-stationary) kernels
$p_x = \delta_0,
q_x = \delta_0,
(q')_x = \ind(x=0) \delta_1 + \ind(x \ne 0)\delta_0$.\footnote{A more general statement is that if $p,p',q,q'$
are Markov kernels and
$\mu, \nu$ are probability measures with
$p \circ \mu = q \circ \mu = \nu$,
if $\mu\{x \in I : p_x = p'_x\} = 1 = \nu\{y \in I : q_y = q'_y \}$,
then
$\mu\{x \in I : (q \circ p)_x = (q' \circ p')_x \} = 1$.
But we only consider the case $\mu = \nu = \mathcal{L}$ here.
}

To describe the correspondence between permutons and Markov kernels,
we need the concept of a \emph{disintegration} of measures, also sometimes called a
\emph{regular system of conditional distributions}.
For a detailed discussion of this (classical) concept,
see for instance Chapter 10.4 of Bogachev \cite{Bog07}. 
For intuition, one can imagine that
the disintegration theorem gives an analogue of Fubini's theorem for an arbitrary probability measure $M$ on $X \times Y$; to integrate a function with respect to $M$,
one can first integrate along each fiber $\{x\} \times Y$
(with respect to an appropriate measure $\mu_x$),
and then integrate over $X$.

\begin{theorem}[Disintegration of measures, Lemma 10.4.3 and Theorem 10.4.5, \cite{Bog07}] \label{disintegration}
    Let $M$ be a probability measure on $X \times Y$, where $X$ and $Y$
    are standard Borel spaces,
    and suppose that $\nu \in \mathrm{Prob}(X)$
    is the pushforward of $M$
    along the projection to $X$.
    Then there exists 
    a Markov kernel $\mu:X \to \mathrm{Prob}(Y)$
    such that
    \[
        M = \int_X (\delta_x \otimes \mu_x) d\nu(x).
    \]
    We call $\mu$ a \emph{disintegration} of $M$. \\
    Moreover, if $\mu':X \to \mathrm{Prob}(Y)$ is another Markov kernel with this property, then we have that
    \[
        \nu(\{x \in X : \mu_x = \mu'_x \})
        =1.
    \]
    That is, the disintegration of $M$ is unique up to $\nu$-a.s. equality.
\end{theorem}

\begin{definition}
    Let $P \in \mathrm{Perm}$ be a permuton. Then the \emph{Markov kernel} $p: I \to \mathrm{Prob}(I)$ 
    \emph{associated to $P$} 
    is the (unique up to $\mathcal{L}$-a.e. equality)
    disintegration of $P$
    guaranteed by applying \cref{disintegration}
    with $X = I = Y$.
\end{definition}

\begin{definition}
    Let $p:I \to \mathrm{Prob}(I)$
    be a Markov kernel
    for which $\mathcal{L}$
    is a stationary distribution.
    Then the \emph{permuton} $P \in \mathrm{Perm}$ \emph{associated to} $p$
    is the measure on $I \times I$ defined by
    \[
        P = \int_I (\delta_x \otimes p_x)dx.
    \]
\end{definition}

\begin{fact} \label{bijection}
   The constructions above define a bijection
   \[
      \{ \mathrm{permutons} \} \leftrightarrow \left\{ \mbox{stationary Markov kernels on } I \right\}/\sim,
   \]
   where the equivalence relation $\sim$ identifies  Markov kernels which are equal on a set of Lebesgue measure 1.
\end{fact}
We provide a proof of \cref{bijection} in the appendix.
\cref{composition is well defined up to ae equivalence} and \cref{bijection} together show that the following  composition operation on permutons is well-defined:
\begin{definition} \label{def:composition}
    Let $P, Q \in \mathrm{Perm}$.
    Then the \emph{composition} $PQ$
    of $P$ and $Q$ is the unique
    permuton corresponding
    to the Markov kernel $p \circ q$,
    where $p,q$ are the Markov kernels associated to $P,Q$ respectively.
\end{definition}
As a notational point, we will generally follow the convention that the Markov kernel associated to a permuton is denoted by the lowercase of the letter denoting the permuton.

We will repeatedly use the following basic facts about the composition operation on $\Perm$, whose proofs we leave as an exercise to the reader:
\begin{prop} \label{basic Perm facts}
    \begin{enumerate}
        \item Composition as defined in \cref{def:composition} makes $\Perm$
        into a semigroup, that is, composition is a well-defined associative binary operation on $\Perm$.
        \item The diagonal permuton $\Delta = \int_{x \in I} (\delta_x \otimes \delta_x)dx$
        is the identity element for $\Perm$, i.e. $P\Delta = P = \Delta P$ for all $P \in \Perm$.
        \item The uniform permuton $U = \int_{x \in I} (\delta_x \otimes \mathcal{L})dx
        = \mathcal{L} \otimes \mathcal{L}$
        is the zero element for $\Perm$, i.e. $UP = U = PU$
        for all $P \in \Perm$.
        In particular, $\Perm$ is \emph{not} a group.
    \end{enumerate}
\end{prop}

\subsection{Familiar special cases}
Let us check that the composition operation we have defined coincides
with familiar composition operations
in three special cases: composition of permutations, composition of functions, and convolution of densities.
Beyond making the composition on $\Perm$ more concrete, we will actually use all three of these special cases for our main theorems.
Recall the definitions of 
$P_\sigma$ and $\gr(F)$ from \cref{sec:notation}.

\begin{prop} \label{familiar special cases}
    Let $P, Q \in \Perm$.
    \begin{enumerate}
        \item Suppose there exist permutations $\sigma, \tau \in S_N$ such that $P = P_{\sigma}, Q = P_{\tau}$. Then $PQ = P_{\sigma \tau}$,
        where $\sigma \tau \in S_N$
        is the usual composition of permutations.
        \item Suppose there exist measure-preserving functions $F,G$
        such that $P = \gr(F), Q = \gr(G)$.
        Then $PQ = \gr(F \circ G)$,
        where $F \circ G$
        is usual function composition.
        \item
        Suppose there exist $f,g \in L^1(I \times I)$
        such that $dP(x,y) = f(x,y)dxdy$
        and $dQ(x,y) = g(x,y)dx dy$.
        Then $d(PQ)(x,y) = (f*g)(x,y) dx dy$,
        where
        \[
            (f*g)(x,y) :=
            \int_{z \in I}
            f(z,y) g(x,z) dz.
        \]
    \end{enumerate}
\end{prop}
\begin{proof}
    (2): If $P = \gr(F), Q = \gr(G)$,
    then $p_x = \delta_{F(x)}, q_x=\delta_{G(x)}$,
    and we compute
    \begin{align*}
        (p \circ q)_x
        &=
        \int_{y \in I} p_y dq_x(y) 
        = \int_{y \in I} \delta_{F(y)} d\delta_{G(x)}(y) 
        = \delta_{F(G(x))},
    \end{align*}
    which is the disintegration for $\gr(F \circ G)$.

    (3): If $dP(x,y) = f(x,y)dxdy$
    and $dQ(x,y) = g(x,y)dxdy$,
    then $dp_x(y) = f(x,y)dy$
    and $dq_x(y) = g(x,y)dy$.
    So, for an arbitrary
    $\phi \in C(I)$, we compute
    \begin{align*}
        \int_{z \in I} \phi(z)d(p \circ q)_x(z)
        &=
        \int_{z \in I} \int_{y \in I} 
        \phi(z) dp_y(z) dq_x(y) 
        = \int_{z \in I} \int_{y \in I} 
        \phi(z) f(y,z) g(x,y) dz dy \\
        &=\int_{z \in I} \phi(z) (f*g)(x,z) dz,
    \end{align*}
    that is,
    $d(PQ)(x,z) = (p \circ q)_x(z)dx = (f*g)(x,z)dzdx$, as desired.

    (1): If $P = P_\sigma$ and $Q = P_\tau$,
    note that $P$ has density
    \[
        f(x,y) := N\sum_{i=1}^N \ind_{I_i}(x) \ind_{I_{\sigma(i)}}(y),
    \]
    where we take $I_i := [(i-1)/N, i/N) \subseteq I$,
    and we have an analogous formula for 
    a density $g(x,y)$ for $Q$.
    Using (3), we then compute a density for $PQ$ by
    \begin{align*}
        (f*g)(x,y) &= \int_{z \in I} f(z,y) g(x,z) dz \\
        &= N^2 \int_{z \in I} \sum_{i=1}^N \sum_{j=1}^N
        \ind(x \in I_i, z \in I_{\tau(i)} = I_j, y \in I_{\sigma(j)} = I_{\sigma(\tau(i))})dz \\
        &= N^2 \sum_{i=1}^N \ind(x \in I_i) \mathcal{L}(I_{\tau(i)}) \ind_{y \in I_{\sigma \tau (i)}}
        = N \sum_{i=1}^N \ind_{I_i}(x) \ind_{I_{\sigma \tau(i)}}(y),
    \end{align*}
    which is the density for $P_{\sigma \tau}$.
\end{proof}

\subsection{Adjoints and a useful formula}
Let us recall the notion of the \emph{transpose} of a permuton $P$,
which we will denote by $P^\dagger$ and which we will also call
the \emph{adjoint} of $P$.

\begin{definition}
    If $P \in \Perm$, the \emph{transpose} or \emph{adjoint} $P^\dagger$ of $P$ is the pushforward of $P$
    by the coordinate-switching map $I \times I \to I \times I, (x,y) \mapsto (y,x)$.
\end{definition}

%Since a permuton is just a coupling of the uniform measure with itself, it is easy to see that any permuton $P$
%gives rise to another permuton $P^\dagger$ just by switching the first and second coordinates.
%That is, considering permutons as measures on $I \times I$, $P^\dagger$ is just the pushforward of $P$
%along the map $I \times I \to I \times I, (x,y) \mapsto (y,x)$.
%We call $P^{\dagger}$ the \emph{adjoint} or \emph{transpose} of $P$.

It is straightforward to check the following facts:
\begin{prop}
    Let $P, Q \in \Perm$.
    \begin{enumerate}
        \item $P^\dagger \in \Perm$.
        \item $(P^\dagger)^\dagger = P$.
        \item $(sP + tQ)^\dagger = sP^\dagger + tQ^\dagger$.
    \end{enumerate}
\end{prop}
The last item is by linearity of pushforward of measures.

Let us describe the relationship between the Markov kernels $p$ and $p^\dagger$
associated to $P$ and $P^\dagger$.
Recall that $p$ is implicitly defined by
\[
   P = \int_{x \in I} (\delta_x \otimes p_x) dx.
\]
Pushing forward both sides of the equation
be the map $(x,y) \mapsto (y,x)$, we then see that
\[
   P^{\dagger} = \int (p_y \otimes \delta_y) dy.
\]
Pairing with a function $f$, we can rewrite this relationship as
\[
   \int_{x \in I} \int_{y \in I} f(x,y) d (p^{\dagger})_x(y) dx = \int_{y \in I} \int_{x \in I} f(x,y) dp_y(x) dy,
\]
and we can abbreviate this relationship by writing
\[
   dP^\dagger(x,y) = d(p^\dagger)_x(y) dx = dp_y(x) dy.
\]
The symmetry between $P$ and $P^\dagger$ also gives us
\[
   dP(x,y) = dp_x(y) dx = d(p^\dagger)_y(x) dy.
\]

With these facts about adjoints, 
we can show the following alternative formula
for composition of permutons, which
resembles the convolution 
formula for absolutely continuous 
permutons from \cref{familiar special cases}.
This formula will be useful
for one of our key lemmas (\cref{characterization of isometries}).
\begin{prop} \label{prop:composition-formula}
   $PQ = \int_{z \in I} ( (q^{\dagger})_z  \otimes p_z ) dz$.
\end{prop}
\begin{proof}
   For a test function $f \in C(I \times I)$, we have
   \begin{align*} 
      \int_{I \times I} f(x,y) d(PQ)(x,y) &=
      \int_{x \in I} \int_{y \in I} f(x,y) d(p \circ q)_x(y) dx \\
      &= \int_{x \in I} \int_{y \in I} f(x,y) d\left( \int_{z \in I} p_z dq_x(z) \right)(y) dx \\
      &= \int_{x \in I} \int_{z \in I} \left(\int_{y \in I} f(x,y) dp_z(y)\right) dq_x(z) dx \\
      &= \int_{z \in I} \int_{x \in I} \left(\int_{y \in I} f(x,y) dp_z(y)\right) d(q^\dagger)_z(x) dz \\
      &= \int_{z \in I} \left[ \int_{(x,y) \in I \times I} f(x,y) d\left( (q^\dagger)_z \otimes p_z \right)(x,y) \right] dz \\
      &= \int_{(x,y) \in I \times I} f(x,y) d\left( \int_{z \in I} (q^\dagger)_z \otimes p_z dz \right)(x,y),
   \end{align*}
   as desired.
\end{proof}

We note that one can use \cref{prop:composition-formula}
to give a quick proof of the expected relation
    $(PQ)^\dagger = Q^\dagger P^\dagger$,
but we leave this as an exercise,
since this will also follow from
our correspondence with operators.

%Our notation and vocabulary strongly suggest attempting to prove the relation $(PQ)^\dagger = Q^\dagger \circ P^\dagger$.

%\begin{prop}
%   $(PQ)^\dagger = Q^\dagger \circ P^\dagger$.
%\end{prop}
%\begin{proof}
%   If we use $\cdot^\dagger$ denote the pushforward of \emph{any} measure $\cdot$ on $I \times I$ by $(x,y) \mapsto (y,x)$,
%   using linearity of pushforward, the fact that for $\mu, \nu \in \Prob(I)$, $(\mu \otimes \nu)^\dagger = \nu \otimes \mu$, and the previous proposition, we have
%   \begin{align*}
%      (PQ)^\dagger &=
%      \left( \int_{z \in I} (q^\dagger)_z \otimes p_z dz \right)^\dagger \\
%      &= \int_{z \in I} \left[(q^\dagger)_z \otimes p_z\right]^\dagger dz \\
%      &= \int_{z \in I} p_z \otimes (q^\dagger)_z dz \\
%      &= Q^\dagger \circ P^\dagger.
%   \end{align*}
%   In the last line we used the previous proposition again, together with $(P^\dagger)^\dagger = P$.
%\end{proof}

\section{Permutons as linear operators} \label{sec:markov operators}

We come to the second key correspondence we use:
permutons can be faithfully encoded as particular kinds of linear operators. This perspective will allow us to use insights from functional analysis, which will be key to our main theorems.

For each permuton $P \in \Perm$,
we define a linear operator $\widehat{P}$
as follows.
Given a permuton $P$ and a bounded measurable function $f: I \to \C$, we define a measurable function $\widehat{P}f: I \to \C$ by
\begin{equation} \label{eq:operator definition}
   (\widehat{P}f)(x) = \int_{y \in I} f(y) d(p^\dagger)_x(y).
\end{equation}
$Pf$ is well-defined up to $\mathcal{L}$-null-sets, since $x \mapsto (p^{\dagger})_x$ is.
The map $f \mapsto \widehat{P}f$ is clearly linear.
The fact that $f = f'$ $\mathcal{L}$-a.s. implies that $Pf = Pf'$ $\mathcal{L}$-a.s. will follow from part (1) of the
following proposition,
%The following proposition 
which collects together 
several facts about $\widehat{P}$
which we will use.
\begin{prop} \label{basic operator facts}
    \begin{enumerate}
    \item For any permuton $P$ and
   any bounded measurable function $f:I \to \C$ and any $p \in [1,\infty]$, we have $\| \widehat{P}f \|_p \le \|f\|_p$,
   where $\| f \|_p := \left( \int_I |f(x)|^p dx \right)^{1/p}$ is the $L^p$ norm with respect to the measure $\mathcal{L}$.
   In particular, $\widehat{P}$ defines a bounded linear operator on each $L^p(I, \mathcal{L})$,
   with operator norm equal to $1$.
   \item For any $f,g \in L^2(I,\mathcal{L})$,
   we have
    \[
      \int_{I \times I} f(x) \overline{g(y)} dP(x,y) = \langle \widehat{P}f, g \rangle,
   \]
   where $\langle \cdot, \cdot \rangle$ denotes the usual inner product on $L^2(I, \mathcal{L})$.
   \item For any $p \in [1,\infty]$,
   the map
   $\Perm \to \mathcal{B}(L^p(I,\mathcal{L})), P \mapsto \widehat{P}$ is injective.
   \item For any $p \in [1,\infty]$,
   the map $\Perm \to \mathcal{B}(L^p(I,\mathcal{L})), P \mapsto \widehat{P}$ is an affine semigroup homomorphism which takes adjoints to adjoints. That is, for any $P, Q \in \Perm$ and $s,t > 0, s+t=1$, we have
   \begin{enumerate}
       \item $\widehat{(sP + tQ)} = s \widehat{P} + t\widehat{Q}$,
       \item $\widehat{(P^\dagger)} = \widehat{P}^\dagger$,
       \item $\widehat{PQ} = \widehat{P} \widehat{Q}$.
   \end{enumerate}
   \end{enumerate}
\end{prop}

Item (1) of \cref{basic operator facts} is 
almost immediate from Jensen's inequality
and the fact that $\mathcal{L}$
is a stationary measure for any permuton $P$.
Most of the rest of the items are relatively straightforward computations, but we provide proofs in the appendix for completeness.

\begin{remark}
    We could equivalently define $\widehat{P}$
    to be the unique operator
    satisfying the equation in item (2)
    of \cref{basic operator facts}.
\end{remark}

\begin{remark}
    Item (3) of \cref{basic operator facts}
    tells us that we can identify $\Perm$
    with a subset of $\mathcal{B}(L^p(I,\mathcal{L}))$.
    Item (4) tells us
    that this identification
    respects convex combinations,
    composition, and adjoints.
    The next proposition
    will tell us
    that we can also identify
    $\Perm$ \emph{as a topological space}
    with a subspace of $\mathcal{B}(L^p(I,\mathcal{L}))$
    with the \emph{weak operator topology}.
\end{remark}

\begin{remark}
    The reader may find the appearance
    of $p^\dagger$ (as opposed to $p$)
    in the definition \eqref{eq:operator definition} a bit strange.
    Indeed, in making this definition, we have adopted a convention
    opposite (or rather, adjoint) to 
    that of most texts on Markov processes.
    We do this primarily
    so that we have the relationship
    $\widehat{PQ} = \widehat{P} \widehat{Q}$
    rather than $\widehat{PQ} =  \widehat{Q}\widehat{P}$.
\end{remark}

\begin{remark}
   It is known that the image
   of $\Perm$ under the map $P \mapsto \widehat{P}$ is precisely the set
   of \emph{Markov operators}
   which have Lebesgue measure as a stationary measure,
   that is, operators satisfying:
   \begin{itemize}
      \item The operator norm of $T : L^p(I, \mathcal{L}) \to L^p(I, \mathcal{L})$ is 1.
      \item $T$ is positive in the sense that if $f \ge 0$, then $Tf \ge 0$.
      \item $T1 = 1$.
   \end{itemize}
   See for instance \cite{Vershik83}
   for a discussion of this.
   Thus, one can think of permutons
   as Markov operators without passing
   through Markov kernels;
   we include the Markov kernel perspective here because we find it more intuitive 
   for some purposes.
\end{remark} 

The following proposition is also well-known in some contexts, and follows from items (1) and (2) of \cref{basic operator facts} together with relatively straightforward approximation arguments;
a proof is provided in the appendix for completeness.
We emphasize that the argument
\emph{does} use the fact that all elements
of $\Perm$ have the same marginals;
if instead of $\Perm$
we were considering all of $\Prob(I \times I)$,
we would not have such a strong result.
Recall the definition of the \emph{weak operator topology}:
\begin{definition}
    Let $X, Y$ be Banach spaces
    and let $T_n, T: X \to Y$
    be bounded linear operators.
    We say that $T_n \to T$
    in the \emph{weak operator topology}
    if for every $f \in X$
    and every continuous linear functional $\phi \in Y^*$ we have
    \[
        \phi(T_n f) \to \phi(Tf) 
    \]
    as $n \to \infty$.\footnote{An attentive reader may be concerned that we have only defined convergence of \emph{sequences} rather than \emph{nets}. In fact, $\Perm$
    is separable and metrizable in its usual topology (see e.g. \cite{Hoppen2013}),
    and so to describe its topology
    it suffices to describe convergence of sequences.}
\end{definition}

\begin{prop}[The usual topology on permutons is the weak operator topology] \label{characterization of weak convergence}
   Let $P_n$ be a sequence of permutons, $P \in \mathrm{Perm}$. Let $\widehat{P}_n$ and $\widehat{P}$ be the associated linear operators. Then the following are equivalent:
   \begin{enumerate}
      \item $P_n \to P$ in the weak-* topology on $C(I \times I)^*$
      (that is, the usual topology on $\Perm$).
      \item For every $p \in (1, \infty)$, $\widehat{P}_n \to \widehat{P}$ in the weak operator topology on $L^p(I, \mathcal{L})$.
      \item For some $p \in [1, \infty]$, $\widehat{P}_n \to \widehat{P}$ in the weak operator topology on $L^p(I, \mathcal{L})$.
      \item For every Borel $A, B \subset I$
    we have
    $P_n(A \times B) \to P(A \times B)$.
   \end{enumerate}
\end{prop}

Note that condition (4) in \cref{characterization of weak convergence} is \emph{stronger} than weak convergence if we consider all of $\Prob(I \times I)$ rather than $\Perm$.
Note also that weak-* convergence of permutons does \emph{not} imply $P_n(S) \to P(S)$
for an arbitrary Borel $S \subset I \times I$;
a counterexample can be given by considering
the identity permutations $\mathrm{id}_n \in S_n$,
noting that 
$P_{\mathrm{id}_n} \to \Delta$, 
and considering the set $S = \{(x,x) : x \in I\} \subset I \times I$.

As we noted in the introduction,
composition of permutons is not jointly continuous with respect to the usual (weak) topology.
However, the equivalence we have established with the weak operator
topology quickly gives us the following:
\begin{prop} \label{continuity fixing one component}
    For any fixed $P \in \Perm$,
    the maps
    \[
        \Perm \to \Perm, Q \mapsto PQ,
        \qquad
        \Perm \to \Perm, Q \mapsto QP
    \]
    are continuous with respect to the usual
    topology on $\Perm$. That is,
    all left- and right- composition
    maps on $\Perm$ are continuous.
\end{prop}
\begin{proof}
    Suppose $Q_n \to Q$ in $\Perm$.
    By \cref{characterization of weak convergence}
    we then have that $\widehat{Q}_n \to \widehat{Q}$ in the weak operator topology.
    Therefore, for any $f,g \in L^2(I,\mathcal{L})$,
    applying weak operator convergence
    to the fixed functions $f$ and $P^\dagger g$
    yields
    \[
        \pair{\widehat{PQ_n} f}{g}
        = \pair{\widehat{Q}_n f}{\widehat{P}^\dagger g}
        \to 
        \pair{\widehat{Q}f}{\widehat{P}^\dagger g}
        = \pair{ \widehat{PQ}f}{g},
    \]
    so $\widehat{P_n Q_n} \to \widehat{PQ}$
    in the weak operator topology,
    hence $PQ_n \to PQ$.
    On the other hand, applying
    weak operator convergence to the fixed functions
    $\widehat{P}f$ and $g$ yields
    \[
        \pair{\widehat{Q_nP}f}{g}
        = \pair{\widehat{Q_n}\widehat{P}f}{g}
        \to \pair{\widehat{Q}\widehat{P}f}{g}
        = \pair{\widehat{QP}f}{g},
    \]
    so $Q_n P \to QP$.
\end{proof}
\cref{continuity fixing one component} will not directly contribute to the proof of \cref{main}.
However, it will be useful 
when we show ``mixing'' for certain
families of permutons (see \cref{sec:mixing}),
which we use for \cref{application}.

\section{Strong convergence of permutons and permutations} \label{sec:strong convergence}

In the introduction we highlighted the fact that there can be no
composition operation $\Perm \times \Perm \to \Perm$ which extends the usual
composition of permutations and which is jointly continuous with respect to the usual topology on permutons. We would
like to find a topology which gives composition better continuity properties.
We saw in the last section that the usual topology on permutons is equivalent to the weak operator topology on their corresponding operators, which suggests exploring other
popular topologies on linear operators.
Perhaps the most familiar topology comes from the \emph{operator norm},
but for our purposes a much weaker topology
known as the \emph{strong operator topology}
will prove more useful.

Recall the following definition from functional analysis:
\begin{definition}
    Let $T_n, T$ be bounded linear operators on $L^2(I, \mathcal{L})$.
    We say that $T_n$ converges to $T$ in the \emph{strong operator topology}
    if for any $f \in L^2(I, \mathcal{L})$
    we have $\|T_nf - Tf\|_2 \to 0$
    as $n \to \infty$.
\end{definition}
This inspires the following definition:
\begin{definition}
    Let $P_n, P \in \mathrm{Perm}$.
    We say that $P_n \to P$ in the \emph{strong topology}
    if 
    $\widehat{P}_n \to \widehat{P}$
    in the strong operator topology.\footnote{Again,
    $\Perm$ is a separable metrizable space when endowed with the strong topology (see e.g. \cite{StrongMetric11}, or note
    that this follows from a general fact about norm-bounded sets of operators on a separable Hilbert space). Therefore, its topology is completely characterized by its convergent sequences.}
\end{definition}

Straightforward functional analysis arguments do show that composition behaves more nicely with respect to the strong topology,
and we take advantage of this in the next
section in order to prove \cref{main}.
The remainder of this section
is dedicated
%we consider strong convergence of \emph{permutations}, %In this case, strong convergence has disadvantages and advantages. 
%On the one hand, we will find that if a sequence of \emph{permutations} converges strongly to a permuton $P$,
%then $\widehat{P}$ must be an isometry. In particular,
%not every permuton can be exhibited as a strong limit of permutations, in contrast with the weak topology (see e.g. Lemma 2.1.3 of \cite{BorgaThesis}).
%On the other hand, we will find that 
%the strong and weak topologies \emph{coincide} when restricted to the strong closure of the image of permutations inside $\Perm$. That is, strong convergence of a permutation is equivalent to weak convergence, plus the condition that the limit permuton is an isometry.
to establishing a very neat characterization of strong convergence of \emph{permutations}. Namely, a sequence of permutations converges strongly to a permuton $P$
if and only if the sequence converges weakly to $P$
and $\widehat{P}$ is an \emph{isometry}.

%However, we do have a nice characterization of the set of strong limit points of permutations;
%moreover, we will see that if a sequence of permutations converges weakly
%to a point in this limit set, it necessarily also converges strongly.

%The limit set will happen to coincide with the set of permutons which are \emph{isometries};

To this end, let us first characterize the set of isometries inside $\Perm$:\footnote{This fact does appear to be known in the literature in various contexts. Some form of it appears in Brown \cite{Brown66}.
See Section 13.2 of \cite{Eisner2015} for a similar theorem regarding the closely related concept of \emph{Markov embeddings}.
Keep in mind the caveat that our definition of the associated operator $\widehat{P}$ is the adjoint 
of the more common definition.
The proof we present here is measure-theoretic
and we have not found a similar proof in the literature, although one may well exist.}
\begin{prop} \label{characterization of isometries}
Let $P \in \mathrm{Perm}$. The following are equivalent.
   \begin{enumerate}
      \item $P$ has a left inverse, i.e. there exists $Q \in \mathrm{Perm}$ such that $QP = \Delta$.
      \item For $\mathcal{L}$-a.e. $x \in I$, $(p^\dagger)_x$ is a Dirac measure
      (i.e. it is supported on a single point).
      \item $P^{\dagger} = \mathrm{gr}(F)$
      for some measure preserving
      function $F:(I, \mathcal{L}) \to (I,\mathcal{L})$.
      \item $\widehat{P}$ is an $L^2$-isometry; that is, $\widehat{P}^{\dagger} \widehat{P} = \mathrm{Id} = \widehat{\Delta}$.
   \end{enumerate}
\end{prop}
\begin{proof}
   First, let us show (1) $\Rightarrow$ (2). 
   Assume (1) holds.
   Let us call a stationary Markov kernel $\mu$ 
    \emph{Dirac} if $\mu_x$ is Dirac for $\mathcal{L}$-a.e. $x$.
   Note that the Markov kernel $\delta$
   associated to the diagonal permuton $\Delta$ is Dirac. We claim that for any stationary Markov kernels $\mu, \nu$, if 
   $\mu \circ \nu$ is Dirac, then so is $\mu$. Once we have this, applying this statement 
   to  
   $\mu = p^\dagger$, $\nu = q^\dagger$ will show that $p^\dagger$ is Dirac,
   since $P^{\dagger}Q^{\dagger} = (QP)^\dagger = \Delta^{\dagger} = \Delta$.

   So let us show that $\mu \circ \nu$ Dirac $\Rightarrow$ $\mu$ Dirac.
   First note that $\{ x \in I : \mu_x \mbox{ is Dirac} \}$ is Borel in $I$.
   This is because, taking $\mathcal{D} = \{ [a,b] : a, b \in \mathbb{Q} \cap [0,1]\}$, we have
   \[
      \{ x \in I : \mu_x \mbox{ is Dirac} \} = \bigcap_{B \in \mathcal{D}} \{x \in I : \mu_x(B) \in \{0,1\} \}
   \]
   is a countable intersection of Borel sets.
   Now suppose for some $x \in I$ that $(\mu \circ \nu)_x$ is a Dirac measure,
   i.e. $(\mu \circ \nu)_x = \delta_{x_0}$ for some $x_0 \in I$.
   Note that $(\mu \circ \nu)_x = \int_{y \in I} \mu_y d\nu_x(y)$. 
   Since Dirac measures are extreme points of $\mathrm{Prob}(I)$,
   $\int_{y \in I} \mu_y d\nu_x(y) = \delta_{x_0}$ implies that $\mu_y = \delta_{x_0}$ for $\nu_x$-a.e. $y \in I$,
   and in particular $\nu_x \{ y \in I : \mu_y \mbox{ is Dirac} \} = 1$.
   Since by assumption this holds for $\mathcal{L}$-a.e. $x \in I$, we have
   \begin{align*}
      \mathcal{L} \left\{ x \in I : \mu_x \mbox{ is Dirac} \right\}
      &= \int_{x \in I} \ind(\mu_x \mbox{ is Dirac}) dx \\
      &= \int_{x \in I} \int_{y \in I} \ind(\mu_y \mbox{ is Dirac}) d\nu_x(y) dx \\
      &= \int_{x \in I} \nu_x \{ y \in I : \mu_y \mbox{ is Dirac} \} dx \\
      &= \int_{x \in I} 1 dx = 1.
   \end{align*}
   In the second line we used the fact that $\mathcal{L} = \nu \circ \mathcal{L}$.
   Thus, $\mu_x$ is $\mathcal{L}$-a.s. Dirac, and we have (1) $\Rightarrow$ (2).
   
   Now let us show (2) $\Rightarrow$ (3).
   Assume (2); we define $F: I \to I$ as follows.
   Since $p^\dagger$ is $\mathcal{L}$-a.s. Dirac, for $\mathcal{L}$-a.e. $x$ there is a unique $y \in I$
   such that $(p^{\dagger})_x = \delta_y$, and we can define $F(x) = y$.
   Then we have $P^{\dagger} = \int_{x \in I} (\delta_x \otimes \delta_{F(x)})dx = \gr(F)$ by construction.
   Moreover, $F$ is measurable since for any Borel $B \subset I$, 
   up to null sets we have $F^{-1}(B) = \{ x \in I: (p^\dagger)_x(B) = 1 \}$,
   and for any fixed $B$, $x \mapsto (p^\dagger)_x(B)$ is measurable.
   The fact that $p^\dagger \circ \mathcal{L} = \mathcal{L}$ tells us exactly
   that $F$ preserves $\mathcal{L}$, since $F_* \mathcal{L} = p^\dagger \circ \mathcal{L}$.
   
   Now let us show (3) $\Rightarrow$ (4). Assume (3), i.e. for some measure-preserving $F:I \to I$,
   we have $P^{\dagger} = \mathrm{gr}(F)$,
   that is, $(p^\dagger)_x = \delta_{F(x)}$ for $\mathcal{L}$-a.e. $x \in I$;
   we want to show 
   $P^{\dagger} P = \Delta$. 
   Recalling Proposition \ref{prop:composition-formula},
   we can compute
   \begin{align*}
      P^{\dagger} P &=
      \int_{z \in I} ((p^{\dagger})_z \otimes (p^\dagger)_z) dz 
      = \int_{z \in I} (\delta_{F(z)} \otimes \delta_{F(z)}) dz \\
      &= \int_{x \in I} (\delta_x \otimes \delta_x) d(F_* \mathcal{L})(x) 
      = \int_{x \in I} (\delta_x \otimes \delta_x) dx = \Delta,
   \end{align*}
   where in the last line we have used that $F$ preserves $\mathcal{L}$.
   
   Lastly (4) $\Rightarrow$ (1) directly by taking $Q = P^{\dagger}$.
\end{proof}
Denote by $\mathrm{End}(I, \mathcal{L})$
the set of all measure-preserving maps $F:I \to I$
(measure-preserving \emph{endomorphisms} of $(I, \mathcal{L})$).
We denote the set of permutons which satisfy any of the equivalent conditions in \cref{characterization of isometries}
by $\gr(\End)^\dagger$, since, as \cref{characterization of isometries} shows, these are exactly the image of
the map
\[
   \mathrm{End}(I, \mathcal{L}) \to \mathrm{Perm}, F \mapsto \mathrm{gr}(F)^\dagger. %\left(\int_{x \in I} (\delta_x \otimes \delta_{F(x)}) dx\right)^{\dagger}.
\]
\cref{characterization of isometries} tells us that we can equivalently think of $\gr(\End)^\dagger$
as the set of transposes of graphs of functions,
as the set of left-invertible permutons, or as the set of permutons whose associated operators are $L^2$-isometries.

We can now characterize strong convergence of permutations.

\begin{prop} \label{characterization of strong convergence}
   Let $\sigma_N \in S_{k_N}$ be a sequence of permutations with $k_N$ a sequence of integers tending to infinity.
   %Denote by $P_N \in \mathrm{Perm}$
   %the corresponding sequence of permutons.
   Then the following are equivalent:
   \begin{enumerate}
      \item $\sigma_N$ converges to a permuton in the strong topology.
      \item $\sigma_N$ converges to an element of $\gr(\End)^\dagger$ in the strong topology.
      \item $\sigma_N$ converges to an element of $\gr(\End)^\dagger$ in the weak topology (i.e., the usual permuton topology).
   \end{enumerate}
\end{prop}
That is, permutations (on divergingly large sets) can only converge strongly to permutons which are isometries. 
But fortunately, if a sequence of permutations converges weakly to an isometry,
it also converges strongly.
\begin{proof}
   Let $\sigma_N$ be a sequence of permutations and set $P_N := P_{\sigma_N}$ to be the permuton associated to the permutation $\sigma_N$
   for each $N$.
   
   First let us show (1) $\Rightarrow$ (2).
   Suppose that for some $P_{\infty} \in \mathrm{Perm}$, $\widehat{P}_N \to \widehat{P}_{\infty}$ in the strong operator topology.
   We will show that $\widehat{P}_{\infty}$ is an isometry (and hence that $P_\infty \in \gr(\End)^\dagger$ by \cref{characterization of isometries}).
   
   To this end, let $u, v \in L^2(I, \mathcal{L})$. By Cauchy-Schwarz,
   the fact that all operators associated to permutons have operator norm $1$, and strong convergence,
      \begin{align*}
      |\langle \widehat{P}_N u, \widehat{P}_N v \rangle - \langle \widehat{P}_{\infty} u, \widehat{P}_{\infty} v \rangle|
      &\le
      |\langle (\widehat{P}_N - \widehat{P}_{\infty}) u, P_N v \rangle| + |\langle \widehat{P}_{\infty} u, (\widehat{P}_N - P_\infty) v \rangle| \\
      &\le
      \|(\widehat{P}_N - \widehat{P}_{\infty}) u\| \|v\| + \| u \| \|(\widehat{P}_N - \widehat{P}_\infty) v \|
      \to 0
   \end{align*}
   as $N \to \infty$, and therefore we have
   \[ \lim_{N \to \infty} \langle \widehat{P}_N u, \widehat{P}_N v \rangle = \langle \widehat{P}_{\infty} u, \widehat{P}_{\infty} v \rangle. \]
   
   Next, since each $P_N$ comes from a permutation, 
   one can check that $P_N^\dagger = (P_{\sigma_N})^\dagger = P_{\sigma_N^{-1}}$.
   By \cref{familiar special cases},
   we then have 
   $P_N^\dagger P_N = P_{\sigma_N^{-1}} P_{\sigma_N} = P_{\mathrm{id}_{k_N}}$, where $\mathrm{id}_{k_N}$
   is the identity element of $S_{k_N}$.
   It is also straightforward to verify
   that $\widehat{P}_{\mathrm{id}_{k_N}}$
   converges to the identity operator as $k_N \to \infty$.
   From this we obtain
   \begin{align*}
      \langle \widehat{P}_N u, \widehat{P}_N v\rangle &= \langle \widehat{P}_N^{\dagger} \widehat{P}_N u, v \rangle 
      = \langle \widehat{P}_{\mathrm{id}_{k_N}} u, v \rangle 
      \to \langle u, v \rangle
   \end{align*}
   as $N \to \infty$.
   Therefore, for each $u, v \in L^2(I, \mathcal{L})$, we have
   $\langle \widehat{P}_{\infty}^{\dagger} \widehat{P}_{\infty} u, v \rangle = \langle \widehat{P}_{\infty} u, \widehat{P}_{\infty} v \rangle = \langle u, v \rangle$,
   so $\widehat{P}_{\infty}^{\dagger} \widehat{P}_{\infty} = \mathrm{Id}$ and $\widehat{P}_\infty$ is an isometry, as desired.
   
   (2) $\Rightarrow$ (1) is immediate, as is (2) $\Rightarrow$ (3). Let us
   show that (3) $\Rightarrow$ (2).
   Assume that $P_N \to P_{\infty} \in \gr(\End)^\dagger$ in the weak operator topology.
   Note that by \cref{characterization of isometries}, $P_\infty^{\dagger} P_\infty = \Delta$.
   Let $v \in L^2(I, \mathcal{L})$, and consider
   \begin{align*}
      \langle (P_N - P_{\infty}) v, (P_N - P_{\infty}) v \rangle &=
      \langle P_N v, P_N v \rangle - \langle P_{\infty} v, P_N v \rangle - \langle P_N v, P_{\infty} v \rangle + \langle P_{\infty} v, P_{\infty} v\rangle.
   \end{align*}
   By weak operator convergence, the middle two terms each converge to $- \langle P_{\infty} v, P_{\infty} v \rangle$.
   Again, since $P_N^{\dagger} P_N = P_{\mathrm{id}_N}$ converges to $I$, the first term converges to $\langle v, v \rangle$,
   and the whole expression converges to $\langle v, v \rangle - \langle P_{\infty} v, P_{\infty} v\rangle$.
   But since $P_{\infty}$ is by assumption an isometry, this is equal to zero. Thus, $P_N \to P_{\infty}$
   in the strong operator topology, as desired.
\end{proof}

Although we will not use it, let us remark that we have a similar theorem to the above if we replace the hypothesis that the $P_N$
come from permutations with the hypothesis that $P_N$ are transposes of graphs of functions:

\begin{prop} \label{characterization of strong convergence for isometries}
    Let $P_N \in \gr(\End)^\dagger$ be a sequence of permutons whose associated operators are $L^2$-isometries. Then the following are equivalent:
    \begin{enumerate}
              \item $P_N$ converges to a permuton in the strong topology.
      \item $P_N$ converges to an element of $\gr(\End)^\dagger$ in the strong topology.
      \item $P_N$ converges to an element of $\gr(\End)^\dagger$ in the weak topology (i.e., the usual permuton topology).
    \end{enumerate}
\end{prop}

Equivalently, \cref{characterization of strong convergence for isometries}
states that $\gr(\End)^\dagger$
is strongly closed, and the
weak and strong topologies coincide \emph{when restricted} to $\gr(\End)^\dagger$.\footnote{Note that however $\gr(\End)^\dagger$ is \emph{not} weakly closed in $\Perm$, nor is it compact.}

\begin{proof}
    The proof is nearly identical to that of \cref{characterization of strong convergence}, and is actually slightly more direct.
    The only difference is that,
    while in the proof of \cref{characterization of strong convergence} we used the fact that $P_N^\dagger P_N = P_{\mathrm{id}_{k_N}}\to \Delta$
    when $k_N \to \infty$,
    here, our assumptions tell us
    that $P_N^\dagger P_N = \Delta$
    for all $N$ (by item (4) of \cref{characterization of isometries}).
    Otherwise the argument is exactly the same.
\end{proof}

We can think of \cref{characterization of strong convergence for isometries} as a generalization of the following fact, proved and used by F\'{e}ray and Rivera-Lopez in \cite{Feray-Rivera-Lopez}: if $F_n,F:[0,1]\to[0,1]$
are measurable measure-preserving functions,
then $\gr(F_n) \to \gr(F)$
if and only if $F_n \to F$ in $L^2(I,\mathcal{L})$.\footnote{The statement in \cite{Feray-Rivera-Lopez} actually shows equivalence to convergence in $L^p(I, \mathcal{L})$ for all $p \in [1,\infty)$.}
To recover this fact from \cref{characterization of strong convergence for isometries}, note that by \cref{characterization of strong convergence for isometries}, $P_n := \gr(F_n) \to \gr(F)=:P$ is equivalent to $\widehat{P}_n^\dagger \to \widehat{P}^\dagger$ in the strong operator topology, that is, $\|\phi(F_n(x)) - \phi(F(x))\|_2 = \|\widehat{P}_n^\dagger \phi - \widehat{P}^\dagger \phi\|_2 \to 0$
for all $\phi \in L^2(I, \mathcal{L})$.
Taking $\phi(x) = x$ therefore gives $\|F_n(x) - F(x)\|_2 \to 0$, as claimed.

\section{Proof of \cref{main} and \cref{corollary diagonal}} \label{sec:proof of main}

\cref{main} will now follow nearly immediately once we recall the following well-known fact from functional analysis:
\begin{prop} \label{mixed weak and strong}
    Let $T_n, T, S_n, S$
    be bounded linear operators 
    on $L^2(I, \mathcal{L})$
    all with operator norm at most 1.
    If $T_n \to T$ in the weak operator topology
    and $S_n \to S$ in the strong operator topology, then $T_n S_n \to TS$
    in the weak operator topology.
\end{prop}
\begin{proof}
    Assume that $T_n \to T$
    weakly and $S_n \to S$ strongly,
    and let
    $f,g \in L^2(I, \mathcal{L})$.
    Then
    \begin{align*}
        |\pair{T_n S_n f}{g} - \pair{TSf}{g}| 
        &\le |\pair{T_n(S_n - S)f}{g}|
        + |\pair{(T_n - T)Sf}{g}| \\
        &\le \|S_n f - S f\| \|g\|
        + |\pair{T_n(Sf)}{g} - \pair{T(Sf)}{g}|,
    \end{align*}
    where in the second line we have
    used Cauchy-Schwarz as well as the uniform operator norm bound on $\|T_n\|$.
    By strong convergence of $S_n$ to $S$
    we have
    \[
        \|S_n f - S f\| \|g\| \to 0,
    \]
    and by weak convergence of $T_n$ to $T$
    we have
    \[
    |\pair{T_n(Sf)}{g} - \pair{T(Sf)}{g}| \to 0
    \]
    as $n \to \infty$,
    and so we are done.
\end{proof}

From this we quickly obtain \cref{main}:

\begin{proof}[Proof of \cref{main}]
    Suppose $\sigma_N, \tau_N$
    are sequences of permutations
    in $S_{k_N}$, with $k_N \to \infty$,
    whose corresponding
    sequences of permutons
    $P_N, Q_N$
    converge (weakly)
    to permutons $P$ and $Q$
    respectively.
    
    First, suppose
    that $Q^\dagger = \mathrm{gr}(F)$
    for some $F \in \mathrm{End}(I,\mathcal{L})$.
    Then, by \cref{characterization of strong convergence}, $Q_n \to Q$
    strongly, that is,
    $\widehat{Q}_n \to \widehat{Q}$
    in the strong operator topology.
    By \cref{characterization of weak convergence}, $\widehat{P}_N \to \widehat{P}$
    in the weak operator topology,
    and so by \cref{mixed weak and strong},
    we have
    that $\widehat{P}_N \widehat{Q}_N \to \widehat{P} \widehat{Q}$ in the weak operator topology,
    and hence $P_N Q_N \to PQ$
    in the usual (weak) topology on $\Perm$
    by \cref{characterization of weak convergence}.

    Now, assume instead that $P = \mathrm{gr}(F)$ for some $F \in \mathrm{End}(I,\mathcal{L})$.
    Note 
    that the adjoint map
    $\cdot^{\dagger}:\Perm \to \Perm$
    is weakly continuous.\footnote{Beware that the adjoint map is \emph{not}
    strongly continuous.}
    Therefore, taking adjoints
    and considering the sequences
    $Q^\dagger_N \to Q^\dagger, P^\dagger_N \to P^\dagger$,
    and $(P_N Q_N)^\dagger = Q^\dagger_N P^\dagger_N$,
    we have the same situation as above.
    Hence we have weak convergence
    \[
    (P_N Q_N)^\dagger =Q^\dagger_N P^\dagger_N \to 
    Q^\dagger P^\dagger = (PQ)^\dagger.
    \]
    But again by weak continuity of the adjoint map, from this we can conclude that
    \[
    P_N Q_N = [Q^\dagger_N P^\dagger_N]^\dagger \to [Q^\dagger P^\dagger]^\dagger = PQ,
    \]
    as desired.
\end{proof}

One should note that from the above one can immediately use induction to get the following,
where we use the notation $\gr(\End)$
for the image of the graph map $\gr:\End(I,\mathcal{L}) \to \Perm$:
\begin{cor} \label{main more complicated}
    Let $\sigma_N^{(1)},...,\sigma_N^{(k)},\tau_N,\rho_N^{(1)},...,\rho_N^{(\ell)} \in S_{k_N}$ be sequences of permutations with $k_N \to \infty$, 
    %and suppose that the corresponding sequences of permutons
    %$P_N^{(1)},...,P_N^{(k)},Q_N,R_N^{(1)},...,R_N^{(\ell)}$
    which
    have (weak) permuton limits
    $P^{(1)},...,P^{(k)},Q,R^{(1)},...,R^{(\ell)}$
    such that all $P^{(i)}, [R^{(j)}]^\dagger \in \gr(\mathrm{End})$. Then the sequence of compositions of permutations
    \[
        \sigma_N^{(1)}\cdots\sigma_N^{(k)}\tau_N\rho_N^{(1)}\cdots\rho_N^{(\ell)}
    \]
    converges (weakly) to the permuton
    \[
    P^{(1)} \cdots P^{(k)}QR^{(1)}\cdots R^{(\ell)}.
    \]
\end{cor}

Another immediate corollary will be relevant to us as we consider \emph{random} permutations:
\begin{cor} \label{main random more complicated}
    Let $\sigma_N^{(1)},...,\sigma_N^{(k)},\tau_N,\rho_N^{(1)},...,\rho_N^{(\ell)} \in S_{k_N}$
    be sequences of \emph{independent random} permutations with $N \to \infty$,
    which
    converge in distribution to random permutons
    $P^{(1)},...,P^{(k)},Q,R^{(1)},...,R^{(\ell)}$.
    Moreover, assume that for each $1 \le i \le k, 1 \le j \le \ell$ we have
    \[
        \Pbb(P^{(i)}, [R^{(j)}]^\dagger \in \gr(\End)) = 1.
    \]
    Then the sequence of random (compositions of) permutations
    \[
    \sigma_N^{(1)}\cdots\sigma_N^{(k)}\tau_N\rho_N^{(1)}\cdots\rho_N^{(\ell)}
    \]
    converges in distribution to
    the random permuton 
    \[
        \tilde{P}^{(1)} \cdots \tilde{P}^{(k)}\tilde{Q}\tilde{R}^{(1)}\cdots \tilde{R}^{(\ell)}.
    \]
    where each $\tilde{P}^{(i)} \sim P^{(i)}, \tilde{Q} \sim Q, \tilde{R}^{(i)} \sim R^{(i)}$,
    and the family
    $\tilde{P}^{(1)},... \tilde{P}^{(k)},\tilde{Q},\tilde{R}^{(1)},..., \tilde{R}^{(\ell)}$
    is independent.
\end{cor}
\begin{proof}
    By Skorokhod's representation theorem,
    for each $i$ there exists
    a single probability space $\Omega_i$
    and a sequence $\tilde{P}^{(i)}_N$
    of $\Perm$-valued random variables on $\Omega_i$
    with each $\tilde{P}^{(i)}_N$ equal in distribution to the permuton associated
    to $\sigma^{(i)}_N$,
    as well as a random variable $\tilde{P}^{(i)}$
    on $\Omega_i$
    such that $\Pbb(\lim_{N \to \infty}\tilde{P}^{(i)}_N = \tilde{P}^{(i)}) = 1$;
    the same holds for $Q_N$ and $R^{(j)}_N$.
    Taking the product of all these probability
    spaces yields a single probability space
    $\Omega$ and a measurable subset
    $\Omega' \subset \Omega$
    with $\Pbb(\Omega') = 1$
    such that on $\Omega'$ we have
    \[
        \forall 1 \le i \le k, \lim_{N \to \infty} \tilde{P}^{(i)}_N = \tilde{P},
        \lim_{N \to \infty}\tilde{Q}_N = \tilde{Q},
        \forall 1 \le j \le l,
        \lim_{N \to \infty} \tilde{R}^{(j)}_N = \tilde{R}^{(i)}.
    \]
    If we set
    \[
        \Omega''
        :=
        \Omega'
        \cap 
        \bigcap_{i=1}^k \{ \tilde{P}^{(i)} \in \gr(\mathrm{End}) \}
        \cap
        \bigcap_{j=1}^\ell
        \{ [\tilde{R}^{(j)}]^{\dagger} 
        \in \gr(\mathrm{End}) \},
    \]
    then $\Pbb(\Omega'') = 1$,
    and by \cref{main more complicated},
    on $\Omega''$ we have
    \[
    \lim_{N \to \infty}
    \tilde{P}_N^{(1)} \cdots \tilde{P}_N^{(k)}\tilde{Q}_N\tilde{R}_N^{(1)}\cdots \tilde{R}_N^{(\ell)}
    =
    \tilde{P}^{(1)} \cdots \tilde{P}^{(k)}\tilde{Q}\tilde{R}^{(1)}\cdots \tilde{R}^{(\ell)}.
    \]
    Thus we have almost-sure convergence of 
    $\tilde{P}_N^{(1)} \cdots \tilde{P}_N^{(k)}\tilde{Q}_N\tilde{R}_N^{(1)}\cdots \tilde{R}_N^{(\ell)}$
    to 
    $\tilde{P}^{(1)} \cdots \tilde{P}^{(k)}\tilde{Q}\tilde{R}^{(1)}\cdots \tilde{R}^{(\ell)}$,
    and therefore we have convergence in distribution.
    But by construction
    $\tilde{P}_N^{(1)} \cdots \tilde{P}_N^{(k)}\tilde{Q}_N\tilde{R}_N^{(1)}\cdots \tilde{R}_N^{(\ell)}$
    is equal in distribution to the 
    random permuton associated to the random permutation
    $\sigma_N^{(1)}\cdots\sigma_N^{(k)}\tau_N\rho_N^{(1)}\cdots\rho_N^{(\ell)}$,
    so we are done. 
\end{proof}

From the above, we can immediately prove \cref{corollary diagonal} from the introduction:
%5which says that if $\sigma_n \to \Delta$,
%then $\sigma_n \tau_n$ and $\tau_n \sigma_n$
%have the same limit as $\tau_n$:
\begin{proof}[Proof of \cref{corollary diagonal}]
    The general statements
    follow immediately
    from the fact that $\Delta=\Delta^\dagger$
    is the graph of the identity function $x \mapsto x$.
    The fact that a uniformly chosen
    $(231)$-avoiding permutation
    converges to $\Delta$
    follows from \cite{MinerPak, HoffmanRizzoloSlivken}.
\end{proof}

\begin{remark}
    A quick functional analysis argument
    shows that composition is jointly \emph{strongly} continuous on $\Perm$, 
    that is, $P_N \to P, Q_N \to Q$
    strongly implies $P_N Q_N \to PQ$
    strongly. If we translate this to \emph{permutations} and use \cref{characterization of strong convergence}, this just tells
    us that if $\sigma_N \to P, \tau_N \to Q$
    and $P,Q \in \gr(\End)^\dagger$,
    then $\sigma_N \tau_N \to PQ$
    and $PQ \in \gr(\End)^\dagger$,
    which already follows from
    \cref{main} and the 
    fact that a composition of
    endomorphisms is itself an endomorphism.
\end{remark}

\section{Compatibility of expectation and composition} \label{sec:expectation}
With \cref{main} proven,
we now turn to the other facts we need
for our main application, \cref{application}.
First we recall the definition of the \emph{expectation} or \emph{mean}
of a \emph{random permuton}.

\begin{definition}
    Let $P : \Omega \to \Perm$
    be a random permuton, that is,
    a permuton-valued random variable.
    The \emph{expectation}
    or \emph{mean} of $P$
    is the unique permuton
    $\E[P] \in \Perm$ satisfying
    \[
        \int f d\E[P] = \E\left[\int fdP\right]
    \]
    for all $f \in C(I \times I)$.
\end{definition}
Note that $\E[P]$ only depends on its \emph{distribution} $\nu \in \Prob(\Perm)$.
The map $\nu \mapsto \E_{P \sim \nu}[P]$
is sometimes called the \emph{barycenter}
or \emph{resultant map}.
We recall that this map is continuous with respect to the usual weak topology on measures:

\begin{prop}[e.g. Proposition 1.1 \cite{Phelps}] \label{barycenter continuity} 
    The map
    \[
        \mathrm{Prob}(\mathrm{Perm}) \to \mathrm{Perm}, \nu \mapsto \E_{P \sim \nu} [P]
    \]
    is well-defined and continuous when $\mathrm{Perm}$ is endowed with the weak topology,
    and $\mathrm{Prob}(\mathrm{Perm}) \subset C(\mathrm{Perm})^*$
    is endowed with the weak-* topology induced by the weak topology on $\mathrm{Perm}$.
\end{prop}

We also highlight the following fact, which we will use for our applications:

\begin{prop} \label{expectation composition}
    Let $P, Q$ be \emph{independent} random permutons. Then \[ \E[PQ] = \E[P] \E[Q]. \]
\end{prop}

This actually follows from a very general fact about products of random linear operators.
For completeness, we give an elementary
proof now:

\begin{proof}
It is clear from the definition 
and item (2) of \cref{basic operator facts}
that if $P$ is a random permuton
and $f,g \in C(I)$, then we have
\[
    \pair{\widehat{\E[P]}f}{g}
    = \E\left[ \pair{\widehat{P}f}{g}\right].
\]
From this it is not hard to show
that the same observation holds for arbitrary
$f,g \in L^2(I,\mathcal{L})$.
It is also not hard to show that $\E[P^\dagger] = \E[P]^\dagger$.

With these observations in mind,
let $P, Q$ be independent random permutons
and let $\{e_i\}_{i \in \N}$
be an orthonormal basis for $L^2(I, \mathcal{L})$,
and let $u,v \in L^2(I, \mathcal{L})$.
We then compute
\begin{align*}
    \pair{\widehat{\E[PQ]} u}{v} &= \E \pair{\widehat{P}\widehat{Q}u}{v} 
    = \E \pair{\widehat{Q}u}{\widehat{P}^\dagger v} \\
    &= \E \left[\sum_i \pair{\widehat{Q}u}{e_i} \pair{e_i}{\widehat{P}^\dagger v} \right] 
    \overset{(*)}{=} \sum_i \E \left[ \pair{\widehat{Q}u}{e_i} \pair{e_i}{\widehat{P}^\dagger v} \right] \\
    &= \sum_i \E \pair{\widehat{Q}u}{e_i} \E \pair{e_i}{\widehat{P}^\dagger v} 
    = \sum_i \pair{\widehat{\E[Q]}u}{e_i} \pair{e_i}{\widehat{\E[P^\dagger]} v} \\
    &= \pair{\widehat{\E[Q]}u}{\widehat{\E[P^\dagger]}v} 
    = \pair{\left(\widehat{\E[P^\dagger]}\right)^\dagger \widehat{\E[Q]}u}{v} \\
    &= \pair{ \widehat{\E[P]}\widehat{\E[Q]} u}{v}.
\end{align*} 
The third line is where we use independence of $P$ and $Q$.
Since $u,v$ were arbitrary, this shows $\widehat{\E[PQ]} = \widehat{\E[P]}\widehat{\E[Q]}$,
hence $\E[PQ] = \E[P] \E[Q]$.
The only questionable equality above is the
one marked $(*)$ in the second line;
%the exchange of the sum and the expectation
%\[
%    \E \left[\sum_i \pair{Qu}{e_i} \pair{e_i}{P^\dagger v} \right]
%    = \sum_i \E \left[ \pair{Qu}{e_i} \pair{e_i}{P^\dagger v} \right],
%\]
so let us show that this holds.

Denote by $\Pi_N$ the orthogonal projection onto the subspace of $L^2(I, \mathcal{L})$ spanned by $e_1,...,e_N$.
Note that for each $N$ we have
\[
    \left|\sum_{i=1}^N \pair{Qu}{e_i} \pair{e_i}{P^\dagger v}\right|
    = \left| \pair{ \Pi_N (Qu) }{ \Pi_N (P^\dagger v)} \right|
    \le \| \Pi_N \| \| Q \| \| u \| \| \Pi_N \| \|P^\dagger \| \|v\| \le \|u\| \|v\|.
\]
Hence, by the dominated convergence theorem, we 
have 
\begin{align*}
    \E \left[\sum_i \pair{Qu}{e_i} \pair{e_i}{P^\dagger v} \right] &=
    \E \left[ \lim_{N \to \infty} \sum_{i=1}^N \pair{Qu}{e_i} \pair{e_i}{P^\dagger v} \right] 
    = \lim_{N \to \infty} \E \left[\sum_{i=1}^N \pair{Qu}{e_i} \pair{e_i}{P^\dagger v} \right] \\
    & = \lim_{N \to \infty} \sum_{i=1}^N \E \left[ \pair{Qu}{e_i} \pair{e_i}{P^\dagger v} \right] 
    = \sum_i \E \left[ \pair{Qu}{e_i} \pair{e_i}{P^\dagger v} \right],
\end{align*}
as desired.
\end{proof}

\subsection{Application: products of separable permutations} \label{sec:first application}

Now we can prove the first part of \cref{application}; this will follow almost
immediately from \cref{main}, \cref{expectation composition},
and work of Maazoun \cite{Maazoun2017}.

\begin{prop} \label{first application}
    For each $N \in \N$, let
    $\sigma_N^{(1)},\sigma_N^{(2)},...$ be an i.i.d. sequence
    of random permutations
    sampled uniformly from the 
    set of all separable permutations in $S_N$.
    For each fixed $k$, as $N \to \infty$, we have
    \[
        \sigma_N^{(1)} \cdots \sigma_N^{(k)} \to P_1 \cdots P_k,
    \]
    in distribution,
    where $P_1,P_2,...$
    is an i.i.d. sequence
    of \emph{separable Brownian permutons} as defined in \cite{Bassino2016}.
    We also have
    \[
        %\E[\sigma_N^{(1)} \cdots \sigma_N^{(k)}] \to 
        \E[P_1 \cdots P_k] = \E[P_1]^k,
    \]
    and in particular, the expectation of 
    the limit permuton $\E[P_1]^k$
    is absolutely continuous with respect to Lebesgue measure, and has density $\gamma \in L^1(I \times I)$
    given by
    \[
        \gamma(x,y)
        =
        \int_{I} \cdots \int_I
        \alpha(z_k,y) \left(\prod_{i=1}^{k-1}
        \alpha(z_i,z_{i+1})\right)
        \alpha(x,z_1) dz_1 \cdots dz_k,
    \]
    where $\alpha = \alpha_{1/2}$
    is as defined in Theorem 1.7 of \cite{Maazoun2017}.
\end{prop}
\begin{proof}
    By \cite{Bassino2016},
    there exists a random permuton $P$,
    called the \emph{Brownian separable permuton},
    such that each $\sigma_N^{(i)}$
    converges to $P$ in distribution
    as $N \to \infty$.
    Theorem 1.3 of Maazoun \cite{Maazoun2017}
    says that there exists some
    \emph{random} measure-preserving
    function $F$ such that
    $P = \gr(F)$ almost surely.
    Therefore, the
    hypotheses of \cref{main random more complicated} are satisfied, and so
    for each $k$ we have
    \[
            \sigma_N^{(1)} \cdots \sigma_N^{(k)} \to P_1 \cdots P_k,
    \]
    in distribution, where the $P_i$
    are independent and have the same distribution as $P$.
    %\cref{barycenter continuity} then tells us that
    %\[
    %\E[\sigma_N^{(1)} \cdots \sigma_N^{(k)}] \to \E[P_1 \cdots P_k]
    %\]
    %as $N \to \infty$.
    \cref{expectation composition}
    then tells us that
    \[
    \E[P_1 \cdots P_k]
    = \E[P_1] \cdots \E[P_k]
    = \E[P]^k,
    \]
    so the remainder
    of \cref{first application} follows from Theorem 1.7 of \cite{Maazoun2017} and \cref{familiar special cases}.
\end{proof}

\section{A criterion for mixing of permutons} \label{sec:mixing}

Recall that we denote by $U$ the ``uniform permuton,'' that is, Lebesgue measure on $[0,1]^2$.
We will call $P \in \mathrm{Perm}$ \emph{mixing} if $P^n \to U$ (weakly) as $n \to \infty$.
The only remaining part of \cref{application} which we have not yet proved is
the fact that if $P$ is the Brownian separable permuton, then the \emph{mean} Brownian separable permuton $\E[P]$ is mixing.

Our main goal in this section is to prove the following sufficient condition for a permuton to be mixing, which we will later apply to the mean separable Brownian permuton:
\begin{prop} \label{mixing criterion}
    Let $P \in \mathrm{Perm}$ be such that $P^\dagger = P$ and for any $A, B \subset I$ Borel with $\mathcal{L}(A), \mathcal{L}(B) > 0$, we have
    $P(A \times B) > 0$. Then $P$ is mixing.
\end{prop}

The overview is as follows. First, if $P$ is symmetric, the Spectral Theorem will tell us that $P^{2n}$ converges to an \emph{idempotent} permuton $\Pi$.
Then, we show that (the operators associated to) idempotent permutons
are all in fact \emph{conditional expectation operators}, that is, for any $\Pi \in \Perm$
idempotent, there exists a sub-$\sigma$-algebra $\mathcal{F}$ of Borel on $I$
such that $\widehat{\Pi}$
is the orthogonal projection
onto the subspace of $\mathcal{F}$-measurable
functions on $I$.
We show that the positive measure condition on $P$ will imply the same on $\Pi$,
and it is not hard to show that
if a conditional expectation operator $\widehat{\Pi}$
satisfies this condition, then 
$\mathcal{F}$ is trivial and so $\Pi = U$.
Thus $P^{2n} \to U$, and a simple
continuity argument will then show $P^n \to U$.

\subsection{Convergence to an idempotent}
   Our first step is to show
   that powers of certain operators converge to idempotents.
   For this, we recall a version of the Spectral Theorem for bounded normal\footnote{Here an operator is \emph{normal} if 
   it commutes with its adjoint, i.e. $T T^\dagger = T^\dagger T$. In particular, if $T=T^\dagger$ is self-adjoint, it is also normal.}
   operators on a Hilbert space $H$.
   We follow Folland \cite{Folland}, Section 1.4.
   Morally, the spectral theorem
   says that the powers $T^n$
   of a normal operator $T$ behave like
   the functions $z \mapsto z^n$
   restricted to $\mathrm{Spec}(T)$;
   in particular, if $\mathrm{Spec}(T) \subseteq [-1,1]$,
   the pointwise convergence of $z^{2n}$
   to the idempotent function $\ind_{\{-1,1\}} =\ind_{\{-1,1\}}^2$ will imply that $T^{2n}$ converges
   to a corresponding idempotent operator, which is our goal.

   \begin{theorem}[Spectral theorem for bounded normal operators, see e.g. Theorem 1.51 in Folland \cite{Folland}] \label{spectral theorem}
      Let $T$ be a bounded operator on a separable Hilbert space $H$ such that $T T^\dagger = T^\dagger T$.
      Define
      \[
         \mathrm{Spec}(T) := \{ \lambda \in \C : (\lambda I - T) \mbox{ is not invertible} \}.
      \]
      Then $\mathrm{Spec}(T)$ is a compact subset of $\C$ contained in $\{ \lambda \in \C : |\lambda| \le \| T \| \}$.
      Denote by $B(\mathrm{Spec}(T))$
      the set of all bounded Borel functions
      on $\mathrm{Spec}(T)$,
      and denote by $\mathcal{B}(H)$
      the set of all bounded linear operators on $H$.
      There exists a map $B(\mathrm{Spec}(T)) \to \mathcal{B}(H),f \mapsto f(T)$
      with the following properties:
      \begin{enumerate}
          \item For all $f, g \in B(\mathrm{Spec}(T))$,
          $(fg)(T) = f(T) g(T)$.
          \item $\overline{f}(T) = f(T)^\dagger$,
          where $\overline{\cdot}$ is complex conjugation.
          \item In the case of the identity map $f(z) = z$, $f(T) = T$.
          \item Suppose that $f_n, f \in B(\mathrm{Spec}(T))$, $f_n \to f$ pointwise,
          and there exists a constant $c \in \R$
          such that $|f_n(z)| \le c$ for all $z \in \mathrm{Spec}(T)$ and $n \in \N$.
          Then $f_n(T) \to f(T)$
          in the strong operator topology.
          \item If $f,g \in C(\mathrm{Spec}(T))$
          are \emph{continuous} and $f(T) = g(T)$,
          then $f = g$.
      \end{enumerate}
   \end{theorem}
   
   We also use the well-known corollary:
   \begin{prop} \label{symmetric operators real spectrum}
      Let $T$ be a self-adjoint bounded operator. Then $\mathrm{Spec}(T) \subseteq [-\|T\|, \|T\|]$.
   \end{prop}
   \begin{proof}
      First, note that by self-adjointness, we must have $\mathrm{Spec}(T) \subseteq \R$;
      otherwise, $z$ and $\overline{z}$ would not be identically equal on $\mathrm{Spec}(T)$,
      and then injectivity of the map 
      $C(\mathrm{Spec}(T)) \to \mathcal{B}(H), f \mapsto f(T)$
      would imply that
      their corresponding operators $T$ and $T^\dagger$ are not equal, contradicting our assumption.
      Next, recall that we always have $\mathrm{Spec}(T) \subseteq \{ \lambda \in \C : |\lambda| \le \| T \| \}$; the intersection of this
      disk with $\R$ is $[-\|T\|, \|T\|]$, and so we are done.
   \end{proof}

   With this, we now easily show that
   \emph{even} powers of any symmetric permuton
   will converge to an idempotent:

   \begin{prop} \label{even powers converge}
       Let $P \in \Perm$ with $P = P^\dagger$.
       Then there exists a permuton $\Pi \in \Perm$ with $\Pi^2 = \Pi$ (i.e. $\Pi$ idempotent) such that
       \[
           P^{2n} \to \Pi.
       \]
       strongly as $n \to \infty$.
   \end{prop}
   \begin{proof}
       Note that since $\widehat{P} = \widehat{P}^\dagger$,
       the Spectral Theorem applies.
       For each $n \in \N$,
       define $f_n \in C(\mathrm{Spec}(\widehat{P})$
       by $f_n(z) = z^{2n}$.
       Also define $f \in B(\mathrm{Spec}(\widehat{P}))$
       by $f(z) = \ind_{\{-1,1\}}(z)$.

       Since $\mathrm{Spec}(\widehat{P}) \subset [-1,1]$
       by \cref{symmetric operators real spectrum},
       we see that the $f_n$ are uniformly bounded
       by the constant $1$, and $f_n \to f$
       pointwise. Therefore, by item (4) of \cref{spectral theorem}, we have strong operator convergence of the corresponding operators
       \[
           \widehat{P}^{2n} = f_n(\widehat{P}) \to f(\widehat{P}) =: \widehat{\Pi}.
       \]
       Since $\widehat{\Pi}$ is the strong limit
       of a sequence of operators associated to permutons,
       it is the operator associated to some $\Pi \in \Perm$. Moreover, since $f$ is $\{0,1\}$-valued,
       \[
           \widehat{\Pi}^2 = f^2(\widehat{P})
           = f(\widehat{P}) = \widehat{\Pi},
       \]
       that is, $\Pi$ is idempotent.
   \end{proof}

\subsection{Characterization of idempotent permutons}
      In this section,
      we establish that idempotent
      permutons are all of a very particular
      form; in particular, their associated operators are all conditional expectation operators.\footnote{This result
      has was shown by Sempi \cite{Sempi2002}
      in the context of copulae,
      and is known in the context
      of Markov operators, 
      see e.g. section 13.3 of \cite{Eisner2015}.
      Again we include
      these proofs for completeness.} This will
      allow us to derive
      a sufficient condition for 
      showing that an idempotent 
      permuton $P$ is equal to $U$,
      which we will ultimately apply to
      prove our mixing statement
      \cref{mixing criterion}.

      Before we show the characterization
      as conditional expectation operators,
      we show a more elementary characterization
      in terms of the associated Markov kernel.
      For intuition, one can think
      of an idempotent permuton
      as coming from an equivalence
      relation on $I$, and the corresponding
      Markov kernel sends each point $x \in I$
      to ``the uniform measure on the equivalence class of $x$.''
      %If one has an equivalence
      %relation on $I$ with only finitely
      %many equivalence classes,
      %one can build a Markov kernel
      %literally as described above,
      %and this gives an idempotent permuton.
      
      \begin{prop} \label{idempotent permutons kernel condition}
      Let $P \in \Perm$. The following are equivalent:
      \begin{enumerate}
          \item $P$ is idempotent, that is, $P^2 = P$.
          \item For $\mathcal{L}$-a.e. $x \in I$,
      \[
         p_x \left( \left\{ y \in I : p_y = p_x \right\}\right) = 1.
      \]
          \item $P$ is idempotent and symmetric, that is, $P^2 = P$ and $P^\dagger = P$.
      \end{enumerate}
      \end{prop}
      \begin{proof}
            First, suppose (2) holds. For $\mathcal{L}$-a.e. $x$ we have
      \[
           p_x \left( \left\{ y \in [0,1] : p_y = p_x \right\}\right) = 1.
      \]
      Hence, for all such $x$, we have
      \begin{align*}
         (p \circ p)_x = \int p_y dp_x(y) 
         = \int p_x dp_x(y) = p_x.
      \end{align*}
      Thus $P^2 = P$ and $P$ is idempotent, as desired.
      
      To see symmetry, we compute
      \begin{align*}
          P &= \int_{x \in I} (\delta_x \otimes p_x)dx 
          = \int_{x \in I} \int_{y \in I} (\delta_y \otimes p_y) dp_x(y) dx \\
          &= \int_{x \in I} \int_{y \in I} (\delta_y \otimes p_x) dp_x(y)dx 
          = \int_{x \in I} \left( \int_{y \in I} \delta_y dp_x(y)\right) \otimes p_x dx \\
          &= \int_{x \in I} (p_x \otimes p_x)dx.
      \end{align*}
      Above, the second inequality is from the fact that $p \circ dx = dx$, and the third inequality is from the condition (2).
      Since each of the measures $p_x \otimes p_x$
      is a symmetric measure on $I \times I$, so is their integral $P$.
      Thus, we have established $(2) \Rightarrow (3)$.

      $(3) \Rightarrow (1)$ is immediate, so let
      us show $(1) \Rightarrow (2)$.
      Suppose $P = P^2$ is an idempotent permuton.
      Fix $A \subset I$ Borel, and note that
      \begin{align*}
         &\int_{x \in I} \int_{y \in I} (p_x(A) - p_y(A))^2 dp_x(y) dx \\
         &= \int_{x \in I} (p_x(A))^2 \int_{y \in I} dp_x(y) dx + \int_{x \in I} \int_{y \in I} (p_y(A))^2 dp_x(y) dx \\
         &- 2 \int_{x \in I} p_x(A) \int_{y \in I} p_y(A) dp_x(y) dx \\
         &= \int_{x \in I} (p_x(A))^2 dx + \int_{y \in I} (p_y(A))^2 dy \\ 
         &  -2 \int_{x \in I} p_x(A) (p \circ p)_x(A) dx \\
         &=0.
      \end{align*}
      The first equality is just using distributivity and linearity.
      In the second equality, we used that %$dx$ is a stationary measure for $P$, i.e. 
      $\int dp_x(y) dx = dy$.
      For the last equality, we use that $p_x = (p \circ p)_x$ for $\mathcal{L}$-a.e. $x$.
      
      Now, since the above is an integral of a nonnegative function, we see that we must have
      \[
         \int_{y \in I} (p_x(A) - p_y(A))^2 dp_x(y) = 0
      \]
      for $\mathcal{L}$-a.e. $x$.
      Since \emph{this} is the integral of a nonnegative function, this means that for all such $x$,
      we have $(p_x(A) - p_y(A))^2 = 0$ for $p_x$-a.e. $y$.
      That is,
      \[
         p_x \left( \left\{ y \in [0,1] : p_y(A) = p_x(A) \right\} \right) = 1.
      \]
      
      Now, let $A$ vary over %$\mathcal{D}$ a countable $\pi$-system
      %generating the Borel $\sigma$-algebra of $I$\footnote{
      %Recall that a $\pi$-\emph{system} is just a nonempty %collection of subsets
      %which is closed under finite intersections. The key property we use is that
      %if two probability measures agree on a $\pi$-system, they also agree
      %on the $\sigma$-algebra generated by that $\pi$-system.
      %In the case $I = [0,1]$ 
      %we can take $\mathcal{D} = \{ [a, b] : a, b \in I \cap \mathbb{Q} \}$;
      %in general, we can take any countable generating set for the $\sigma$-algebra
      %and adjoin to it the collection of all finite intersections of its members, which will still be countable.}.
      all elements of $\mathcal{D} = \{ [a,b] : a,b \in \mathbb{Q} \cap [0,1] \}$.
      Since this is a countable collection, we see that for $\mathcal{L}$-a.e. $x \in [0,1]$, 
      we have
      \[
         p_x \left( \left\{ y \in [0,1] : p_y(A) = p_x(A) \mbox{ } \forall A \in \mathcal{D} \right\} \right) = 1.
      \]
      and hence
      \[
         p_x \left( \left\{ y \in [0,1] : p_y = p_x\right\} = 1 \right) = 1,
      \]
      that is, $P$ satisfies (2), as desired.
   \end{proof}

      Next, let us
      recall the concept of \emph{conditional expectation operators
      with respect to a sub-$\sigma$-algebra}. 

      \begin{definition}
      Let $(\Omega, \mathcal{F}, \mu)$ be a probability space and $\mathcal{F}' \subset \mathcal{F}$
      is a sub-$\sigma$-algebra. 
      For any $f \in L^1(\Omega, \mathcal{F}, \mu)$,
      the \emph{conditional expectation} of $f$
      with respect to $\mathcal{F}'$
      is the unique (up to $\mu$-a.s. equality)
      function $\E[f | \mathcal{F}'] \in L^1(\Omega, \mathcal{F},\mu)$
      satisfying
      \begin{enumerate}
          \item $\E[f | \mathcal{F}']$ is measurable with
          respect to $\mathcal{F}'$,
          \item for every $A \in \mathcal{F}'$ we have
      \[
         \int_\Omega \ind_A \E[ f | \mathcal{F}'] d\mu = 
         \int_\Omega \ind_A f d\mu.
      \]
      \end{enumerate}
      \end{definition}
      For a proof of existence and uniqueness of conditional expectations, see e.g. Section 34 of Billingsley \cite{Billingsley}.
      It is straightforward to verify that,
      for each fixed sub-$\sigma$-algebra $\mathcal{F}'$,
      the map $f \mapsto \E[f | \mathcal{F}']$
      defines a bounded linear operator on $L^1(\Omega,\mathcal{F},\mu)$.
      Note also that if $f$
      is already $\mathcal{F}'$-measurable,
      then $\E[f|\mathcal{F}'] = f$.

      Now we will show that each idempotent
      permuton is in fact a conditional expectation operator:

    \begin{prop} \label{idempotent permutons are condition expectations}
         Let $P \in \mathrm{Perm}$ and suppose $P^2 = P$. 
         Then there exists a sub-$\sigma$-algebra $\sigma(P)$
         of the Borel $\sigma$-algebra $\mathcal{B}$ on $I$ such that
         for every $f \in L^1(I, \mathcal{L})$, we have
         \[
            \widehat{P}f = \E[ f | \sigma(P) ]
         \]
         $\mathcal{L}$-a.s. That is, the operator $\widehat{P}$ is equal to the conditional expectation operator with respect to $\sigma(P)$.
      \end{prop}
      \begin{proof}
      First, let us construct $\sigma(P)$.
      Let $\mathcal{A}$ be the $\sigma$-algebra on $\Prob(I)$
      generated by all sets of the form
      \[
          \{ \mu \in \Prob : \mu(A) \in B \},
      \]
      where $A \subseteq I$ and $B \subseteq \R$ are Borel.
      Then, for any $P \in \Perm$, by definition of a Markov kernel,
      the map $p: (I, \mathcal{B}) \to (\Prob(I), \mathcal{A}), x \mapsto p_x$
      is measurable.
      
      Now, given an idempotent permuton $P \in \mathrm{Perm}$, we define $\sigma(P)$ to be the 
      sub-$\sigma$-algebra
      of $\mathcal{B}$ ``generated'' by $p: I \to \mathrm{Prob}(I)$, i.e.
      \[
         \sigma(P) := \{ p^{-1}(B) : B \in \mathcal{A} \}
         = \{ \{x \in I : p_x \in B \} : B \in \mathcal{A} \}.
      \]
      
         Let $f \in L^1(I, \mathcal{L})$.
         By uniqueness of conditional expectation, it suffices to prove that:
         \begin{enumerate}
            \item $\widehat{P}f$ is $\sigma(P)$-measurable.
            \item For any $A \in \sigma(P)$ we have $\int_{x \in A} (\widehat{P}f)(x) dx = \int_{x \in A} f(x) dx$.
         \end{enumerate}
        We several times below will use that by \cref{idempotent permutons kernel condition}, $P$ is symmetric,
        and hence $(p^\dagger)_x = p_x$ $\mathcal{L}$-a.s.
        
      To see (1), first assume that $f = \ind_A$
      is the indicator of some Borel set $A \subseteq I$. In that case,
      for any $-\infty<a<b<\infty$ we have
      \[
         (\widehat{P}f)^{-1}(a,b) = \left\{ x \in I : \int_{y \in I} \ind_A(y) dp_x(y) \in (a,b) \right\}
         = p^{-1} \{ \nu \in \Prob(I) : \nu(A) \in (a,b) \} \in \sigma(P),
      \]
      so $\widehat{P}f$ is indeed $\sigma(P)$-measurable.
      Linearity of $\widehat{P}$ then shows that $\widehat{P}f$
      is $\sigma(P)$-measurable whenever $f$ is a 
      simple function (sum of indicator functions).
      In the case of general $f \in L^1(I, \mathcal{L})$, we can find simple functions $f_n \to f$ in $L^1(I, \mathcal{L})$,
      whence $\widehat{P} f_n \to \widehat{P} f$ in $L^1(I, \mathcal{L})$, in particular, converging pointwise $\mathcal{L}$-a.e.
      along a subsequence. By our previous argument, each $\widehat{P} f_n$ is $\sigma(P)$-measurable,
      so $\widehat{P}f$ is as well, as the pointwise limit of a sequence of $\sigma(P)$-measurable functions.
      
      Next, let us show (2). Fix $A \in \sigma(P)$ and note that there must
      exist some $B \in \mathcal{A}$ such that $A = \{ x \in I : p_x \in B \}$.
      We compute
      \begin{align*}
         \int_{x \in A} (\widehat{P}f)(x) dx
         = \int_{x \in A} \int_{y \in I} f(y) dp_x(y) dx 
         = \int_{p_x \in B} \int_{y \in I} f(y) dp_x(y) dx.
      \end{align*}
      
      On the other hand, using $p \circ dx = dx$, we compute
      \begin{align*}
         \int_{x \in A} f(x) dx = \int_{x \in I} \int_{y \in A} f(y) dp_x(y) dx 
         = \int_{x \in I} \int_{p_y \in B} f(y) dp_x(y) dx.
      \end{align*}
      
      But since $P$ is an idempotent permuton, 
      by \cref{idempotent permutons kernel condition}, for $\mathcal{L}$-a.e. $x$ we have $p_y = p_x$ for $p_x$-a.e. $y$,
      and therefore
      both integrals above are equal to the integral
      \[
         \int_{(x,y) : p_x = p_y,p_x \in B, p_y \in B} f(y) dp_x(y) dx.
      \]
      %In other words, since $P$ is idempotent, if we consider a Markov chain $(X_n)_{n \in \N}$
      %with transition kernel $P$ and $X_0 \sim \mathcal{L}$, then almost surely we have
      %$p_{X_n} = p_{X_0}$ for all $n \in \N$. If we interpret the integrals above in this context,
      %we see that
      %\[
      %   \int_{x \in A} (Pf)(x) dx = \E[ f(X_2) \ind_{p_{X_0} \in B} ]
      %\]
      %and
      %\[
      %   \int_{x \in A} f(x) dx = \E[ f(X_2) \ind_{p_{X_1} \in B} ].
      %\]
      %The above observation tells us that both of these are equal to
      %\[
      %   \E[ f(X_2) \ind_{p_{X_0} = p_{X_1} \in B} ]
      %\]
      %and in particular are equal. 
      So we have shown (2), as desired.
      \end{proof}
     
     From here, it is easy to establish a fact we use later:
     \begin{prop} \label{characterization of uniform}
         Let $\Pi = \Pi^2$ be an idempotent permuton, and suppose that for any $A, B$ Borel with
         $\mathcal{L}(A), \mathcal{L}(B) > 0$ we have $\Pi(A \times B) > 0$. Then $\Pi = U$.
     \end{prop}
     \begin{proof}
         We show the contrapositive.
         Let $\sigma(\Pi)$ be the $\sigma$-algebra associated to $\Pi$,
         as in \cref{idempotent permutons are condition expectations}.
         Suppose that $\Pi \ne U$. 
         Then there exists some $A \in \sigma(\Pi)$ with $\mathcal{L}(A) \notin \{0,1\}$.\footnote{If all $A \in \sigma(\Pi)$
         have $\mathcal{L}(A) \in \{0,1\}$, then $\widehat{\Pi}$
         is projection onto ($\mathcal{L}$-a.e.) constant functions,
         which is precisely the operator $\widehat{U}$.}
         Since $\ind_A$
         is $\sigma(P)$-measurable, $\widehat{\Pi} \ind_A = \E[\ind_A | \sigma(\Pi)] = \ind_A$. Therefore
         \begin{align*}
             \Pi(A \times A^c) = \langle \Pi \ind_A, \ind_{A^c} \rangle  
             = \langle \ind_A, \ind_{A^c} \rangle = 0,
         \end{align*}
         i.e. there exist $A$ and $B=A^c$ with positive measure such that $\Pi(A \times B) = 0$.
     \end{proof}

\subsection{Proof of \cref{mixing criterion}}

In this subsection we prove the mixing criterion \cref{mixing criterion}.
We first show the following ``absolute continuity'' statement, which states that if $Q^n$ converges to an idempotent $\Pi$
and $Q$ transports some of $A$ to $B$ ($Q(A \times B) > 0$), then $\Pi$ also transports some of $A$ to $B$.

\begin{prop} \label{absolute continuity}
    Let $Q \in \Perm$
    and suppose that $Q^n \to \Pi$
    with $\Pi^2 = \Pi$ (i.e. $\Pi$ idempotent).
    Then, for any $A, B$ Borel,
        \[
            Q(A \times B) > 0
            \Rightarrow
            \Pi(A \times B) > 0.
        \]
\end{prop}
\begin{proof}
    First, note that,
    since left multiplication by a fixed permuton is continuous
    (\cref{continuity fixing one component}), we have
    \[
        Q\Pi = Q \lim_{n \to \infty} Q^n
        = \lim_{n \to \infty} Q^{n+1} = \Pi.
    \]
    Suppose that for some $A, B \subset I$ Borel we have
    $Q(A \times B) > 0$.
    Since
    \begin{align*}
        Q(A \times B) &=
        \int_{x \in I} \int_{y \in I} \ind_A(x) \ind_B(y) d(q_x)(y) dx \\
        &=
        \int_{x \in A} q_x(B) dx,
    \end{align*}
    there exists $\epsilon > 0$ and $A' \subset A$
    Borel with $\mathcal{L}(A') > 0$ such 
    that for all $x \in A'$ we have
    \[
        q_x(B) \ge \epsilon.
    \]

    Now, using $Q \Pi = \Pi$,
    we compute
    \begin{align*}
        \Pi(A \times B) &= (Q\Pi)(A \times B) \\
        &= \int_{x \in I} \int_{y \in I}
        \ind_A(x) \ind_B(y) d(q \circ \pi)_x(y) dx \\
        &= \int_{x \in I} \int_{y \in I} \int_{z \in I}
        \ind_A(x) \ind_B(z) dq_y(z) d\pi_x(y) dx \\
        &\ge
        \int_{x \in I} \int_{y \in I} \int_{z \in I}
        \ind_{A'}(x) \ind_{A'}(y) \ind_B(z) dq_y(z) d\pi_x(y) dx \\
        &= \int_{x \in I} \int_{y \in I}
        \ind_{A'}(x) \ind_{A'}(y) q_y(B) d\pi_x(y) dx \\
        &\ge
        \epsilon \int_{x \in I} \int_{y \in I} \ind_{A'}(x) \ind_{A'}(y) d\pi_x(y) dx \\
        &= \epsilon \pair{\widehat{\Pi} \ind_{A'}}{\ind_{A'}} \\
        &= \epsilon \pair{\widehat{\Pi} \ind_{A'}}{\widehat{\Pi} \ind_{A'}} \\
        &= \epsilon \|\widehat{\Pi} \ind_{A'} \|_2^2 > 0.
    \end{align*}
    In the second-to-last line, we used that
    $\Pi$ is idempotent and symmetric.
    In the last line, we used that $\widehat{\Pi} \ind_{A'} \ne 0$,
    which is the case because
    \[
        \pair{\widehat{\Pi} \ind_{A'}}{1} =\pair{\ind_{A'}}{\widehat{\Pi} 1} =
        \pair{\ind_{A'}}{1}
        = \mathcal{L}(A') > 0.
    \]
\end{proof}

Next, since we want to apply the above to $Q=P^2$, we show that our assumption
on $P$ from \cref{mixing criterion}
implies the same property for $P^2$.

\begin{prop} \label{P to P2}
    Let $P \in \Perm$
    and suppose that for every
    $A,B \subset I$ Borel with $\mathcal{L}(A),\mathcal{L}(B)>0$
    we have $P(A \times B) > 0$.
    Then for every $A,B \subset I$ Borel with $\mathcal{L}(A),\mathcal{L}(B)>0$,
    we also have $P^2(A \times B) > 0$.
\end{prop}
\begin{proof}
    Let $A,B \subset I$ be Borel with positive measure.
    The same argument as in the proof
    of \cref{absolute continuity}
    shows us that there exist $\epsilon > 0$
    and $A' \subset A$
    Borel with positive measure such that
    $p_x(B) \ge \epsilon$
    for all $x \in A'$.

    A similar computation as in the proof
    of \cref{absolute continuity}
    then shows that
    \begin{align*}
        P^2(A \times B)
        \ge
        \epsilon \int_{x \in I}
        \int_{y \in I}
        \ind_A(x) \ind_{A'}(y)
        dp_x(y) dx \\
        = \epsilon P(A \times A') > 0,
    \end{align*}
    as desired.
\end{proof}

At last we can prove our mixing criterion:
\begin{proof}[Proof of \cref{mixing criterion}]
    Let $P$ be a symmetric permuton such that
    for every Borel $A, B$ with $\mathcal{L}(A), \mathcal{L}(B) > 0$, $P(A \times B)>0$. By \cref{P to P2},
    $Q = P^2$ also has this property.

    Since $P$ is a symmetric permuton, 
    by \cref{even powers converge} 
    $(P^2)^n = Q^n$
    converges to an idempotent permuton $\Pi$.
    By \cref{absolute continuity}, we see that such that
    $\Pi(A \times B) > 0$ for all Borel $A,B$
    with $\mathcal{L}(A), \mathcal{L}(B) > 0$.

    But then by \cref{characterization of uniform},
    we have that $\Pi = U$.
    Since left multiplication by $P$
    is a (weakly) continuous map on $\Perm$ (\cref{continuity fixing one component}),
    we therefore see that
    \[
        \lim_{n \to \infty} P^{2n+1}
        = P \lim_{n \to \infty} (P^2)^n
        = PU = U.
    \]
    Since $P^{2n+1}$ and $P^{2n}$ both converge to $U$
    as $n \to \infty$, we then have $P^n \to U$
    as $n \to \infty$, as desired.
\end{proof}

\section{Application: mixing in mean of separable permutations} \label{sec:second application}

Although \cref{first application} gives us a formula for computing the mean of the limit of the composition of any finite number of independent separable permutations,
in practice, the computation
becomes harder and harder the larger the number of compositions becomes. However, in the limit, we have the following qualitative result: as the number of compositions tends to infinity, the mean limit permuton tends to the uniform permuton;
this will complete the proof of \cref{application}.

\begin{prop} \label{separable mixing}
    Let $P_1, P_2, P_3,...$ be an i.i.d. sequence of Brownian separable permutons. Then
    \[
        \lim_{k \to \infty} 
        \E[P_1 \cdots P_k]
        = U.
    \]
\end{prop}
\begin{proof}
    By \cref{expectation composition}, we know that
    \[
        \E[P_1 \cdots P_k]
        = \E[P_1]^k.
    \]
    Thus it suffices
    to show that $\E[P_1]$
    satisfies the hypotheses of
    \cref{mixing criterion}.
    But Theorem 1.7 of Maazoun \cite{Maazoun2017}
    gives a formula for the density
    of $\E[P_1]$ with respect
    to Lebesgue measure which 
    is symmetric in $x$ and $y$
    and $\mathcal{L}^2$-a.e. strictly positive,
    and so it is immediate to verify
    the hypotheses of \cref{mixing criterion}.
\end{proof}

\begin{proof}[Proof of \cref{application}]
    \cref{application} follows immediately from
    \cref{first application} and \cref{separable mixing}.
\end{proof}

%Although this
%does give us a more
%concrete description of
%our limit, much combinatorial information is lost when passing from a random permuton to its expectation.
%Subpattern densities are \emph{not}
%linear functionals on $\Perm$,
%and hence one cannot extract expected subpattern densities from $P$ from $\E[P]$.\footnote{In fact, the expected subpattern densities of a random permuton $P$ entirely determine its distribution; see Theorem 2.5 and Proposition 2.9 of \cite{Bassino2020}.}
%However, we note that from our arguments we can obtain the following qualitative \emph{negative} result.

We have now established all of the main theorems mentioned in the introduction. 
Let us lastly give an example of how
the methods used to prove \cref{separable mixing}
can be applied to the question of which families of random permutons are closed under independent compositions. Specifically, we will show that one natural family of random permutons, the \emph{biased separable Brownian permutons}, are \emph{not} closed under independent compositions.

For each $p \in (0,1)$,
let $\mu^p \in \Prob(\Perm)$
be the (distribution of the) \emph{$p$-biased separable Brownian permuton}
as defined in Section 5 of \cite{Bassino2020}.
Set $\mathcal{S} = \{ \mu^p : p \in (0,1)\}$.
Since $\mathcal{S}$
is a natural $1$-parameter family of random permutons,
one might wonder if $\mathcal{S}$ forms a semigroup, that is, if any product of an independent $p_1$-biased Brownian separable permuton and a $p_2$-biased Brownian separable permuton is a $p_3$-biased separable permuton for some $p_3 \in (0,1)$.
The next proposition states that this is \emph{not} the case.

\begin{prop}
    The set $\mathcal{S}$
    is \emph{not} closed under independent products.
    In fact, there exists $p \in (0,1)$
    such that if $P_1, P_2 \sim \mu^p$
    are independent,
    the distribution of $P_1 P_2$ does not lie in $\mathcal{S}$.
\end{prop}
\begin{proof}
    Suppose to the contrary that for every $p \in (0,1)$, the independent product of $\mu^p$
    with itself lies in $\mathcal{S}$.
    By \cref{expectation composition},
    this means in particular that for any $p \in (0,1)$,
    \[
    \E[\mu^{p}]^2 \in \E[\mathcal{S}] := \{ \E[\mu^p] : p \in (0,1)\}.
    \]
    By induction $\E[\mu^p]^{2^n} \in \E[\mathcal{S}]$
    for all $n \in \N$.
    By Theorem 1.7
    in \cite{Maazoun2017},
    we see that $\E[\mu^p]$
    is symmetric and
    has strictly positive density almost everywhere on $I^2$.
    Therefore by \cref{mixing criterion},
    $\lim_{n \to \infty} \E[\mu^p]^{2^n} = \lim_{n \to \infty} \E[\mu^p]^n = U$.
    But we claim that 
    $U$ is not in the weak closure of $\E[\mathcal{S}]$;
    once we have this, we will have a contradiction, and the proof is complete.

    For this, note that
    the map
    \[
    p \mapsto \begin{cases}
        \Delta & \mbox{ if } p = 0 \\
        \mu^p & \mbox{ if } p \in (0,1) \\
        \int(\delta_x \otimes \delta_{1-x})dx & \mbox{ if p = 1}
    \end{cases}
    \]
    is continuous.
    Continuity in $p \in (0,1)$ is clear from Definition 5.1 and Theorem 2.5 of \cite{Bassino2020}.
    The fact that the above are the correct limits at the endpoints $\{0,1\}$ is discussed for instance in Section 7.2.1.2 of
    \cite{MaazounThesis}.
    Thus, the above
    and \cref{barycenter continuity} show that
    $\E[\mathcal{S}] \cup \{\Delta, \int (\delta_x \otimes \delta_{1-x})dx\}$
    is a weak closed subset of $\Perm$, as the continuous image of the compact set $[0,1]$.
    Referring to the densities given by Theorem 1.7 of \cite{Maazoun2017},
    we see that $U$ does not lie in this set, we are done.\footnote{To see that there is no $p \in (0,1)$ with $\E\mu^p = U$,
    it suffices to show that each $\alpha_p$
    is not almost-everywhere equal to the constant function $1$.
    The integral description of $\alpha_p$
    in Theorem 1.7 of \cite{Maazoun2017}
    makes it easy to find a point where $\alpha_p$
    is differentiable and has nonzero derivative,
    and this is sufficient.
    }
\end{proof}

It in fact seems plausible that $\mathcal{S} \mathcal{S}$ is \emph{disjoint} from $\mathcal{S}$, although it is not as clear how to show this; sufficient understanding of convolutions of the functions $\alpha_p$
given in Theorem 1.7 of \cite{Maazoun2017}
could achieve this. This example raises the question of whether we should expect any of the natural random permuton families from combinatorics to be closed under independent composition; understanding this question will clarify the relationship between subpattern densities of permutations and composition.

\section{Remarks on the algebra of the monoid $\Perm$} \label{sec:algebra}
Here, we collect most of the algebraic facts we have shown about $\Perm$ into one proposition,
for the convenience of the interested reader.
Below, for brevity, we denote by $\mathrm{End} \subset \Perm$
 the image of $\mathrm{End}(I,\mathcal{L})$
under the natural ``graph'' map $\gr: \mathrm{End}(I,\mathcal{L}) \to \Perm$ (see \cref{sec:strong convergence}).
Similarly we denote by $\mathrm{Aut} \subset \Perm$
the image under $\gr$ of the set $\mathrm{Aut}(I,\mathcal{L})$ of \emph{invertible} measure-preserving endomorphisms of $(I, \mathcal{L})$.

The following observations are either straightforward to show or have already been shown above:
\begin{prop}
    The following hold:
    \begin{enumerate}
    \item $\Perm$ is a monoid with identity element $\Delta$. $U$ is a zero element for $\Perm$.
    \item $P \in \Perm$
    has a right inverse if and only if
    $P \in \End$, and in this case
    $PP^\dagger = \Delta$.
    \item $P \in \Perm$
    has a left inverse if and only if
    $P \in \End^\dagger$, i.e. $P^\dagger \in \End$, and in this case
    $P^\dagger P = \Delta$.
    \item $P$ has a two-sided
    inverse if and only if $P \in \Aut$,
    and in this case 
    the inverse is unique and $P^{-1} = P^\dagger$.
    \item $\Aut = \End \cap \End^\dagger$.
    \item $P$ is left-invertible if and only if
    $\widehat{P}$ is an $L^2$-isometry.
    \item $P$ is invertible
    if and only if $\widehat{P}$ is unitary.
    \item Suppose $\Pi \in \Perm$
    is idempotent, that is, $\Pi \Pi = \Pi$.
    Then $\Pi$ is also \emph{symmetric},
    that is, $\Pi = \Pi^\dagger$.
    Hence $\widehat{\Pi}$ is an \emph{orthogonal projection}.
    \item If $\Pi \in \Perm$
    is idempotent, there exists
    some sub-$\sigma$-algebra
    $\mathcal{F}$ of the Borel $\sigma$-algebra on $I$
    such that for all $f \in L^2(I,\mathcal{L})$ we have 
    $\Pi f = \E[ f | \mathcal{F} ]$.
    Conversely every sub-$\sigma$-algebra
    has conditional expectation operator
    arising from an idempotent permuton.
    Two sub-$\sigma$-algebras
    correspond to the same permuton
    if and only if they have the same completion.
    \item With the weak topology, $\Perm$
    is a \emph{compact separable metrizable semitopological semigroup},
    which is to say that for each
    fixed $P \in \Perm$, the maps
    \[
        \Perm \to \Perm, Q \mapsto PQ,
        \qquad
        \Perm \to \Perm, Q \mapsto QP
    \]
    are continuous.
    However, the composition map
    $\Perm \times \Perm \to \Perm$
    is not (jointly) continuous with respect to this topology.
    \item With the strong topology,
    $\Perm$ is a \emph{non-compact
    separable metrizable topological semigroup},
    which is to say that the
    multiplication map
    \[
        \Perm \times \Perm \to \Perm,
        (P,Q) \mapsto PQ
    \]
    is (jointly) continuous.
    \item $\End^\dagger \subset \Perm$
    is strongly closed but weakly dense.
    $\End$ and $\Aut$ are also weakly dense.\footnote{Strong closure is from \cref{characterization of strong convergence for isometries}.
    One way to see weak density of $\Aut$
    is to use weak density of permutations (which follows from e.g. Lemma 4.2 of Hoppen et al. \cite{Hoppen2013}).
    Then, instead of encoding a sequence of permutations according to \eqref{eq:permutation to permuton},
    one can encode a permutation $\sigma$ by the automorphism of $I$
    which break up $I$ into $N$ equal length intervals and reorders them according to $\sigma$;
    it is not hard to show that this construction is weakly asymptotic to $P_\sigma$ as $N \to \infty$.
    }
    \item The restriction of the weak topology to $\End^\dagger$
    is equal to the restriction
    of the strong topology to $\End^\dagger$.\footnote{The proof of this
is almost identical to the proof of \cref{characterization of strong convergence}.}
    So the restricted composition maps
    \[
        \End^\dagger \times \End^\dagger \to \End^\dagger,
        \qquad
        \Aut \times \Aut \to \Aut
    \]
    are (jointly) weakly continuous.
    \item The adjoint map $\Perm \to \Perm, P \mapsto P^\dagger$
    is weakly continuous but \emph{not}
    strongly continuous.
    Thus the restricted composition map
    \[
        \End \times \End \to \End
    \]
    is also (jointly) weakly continuous.
    \item The composition structure we defined is the unique operation on $\Perm$ which agrees with composition of permutations on $\bigcup_{n=1}^\infty \{P_\sigma : \sigma\in S_n\}^2$, which is jointly (weakly) continuous 
    when restricted to
    \[
        \left[\left(\bigcup_{n=1}^\infty \{P_\sigma : \sigma \in S_n\}\right) \cup \End\right] \times \Perm \to \Perm,
    \]
    and for which composition is weakly continuous with one component fixed.
    \item $\Perm$ is convex,
    and $\End \cup \End^\dagger$
    lies inside the set of extreme points.
    Hence the set of extreme points
    of $\Perm$ is (weakly) dense in $\Perm$.\footnote{$\Perm$
    has many other extreme points outside
    of $\End \cup \End^\dagger$,
    although this is much harder to see.}
    \end{enumerate}
\end{prop}

A thing to note about the above
is that even though the
definition of a permuton
does not explicitly require
any compatibility with the adjoint operator
or the $L^2$-metric, we see compatibilities
with the $L^2$-metric at least at algebraically special elements.
For instance, all idempotents are \emph{orthogonal} projections,
and all left-invertible elements are
\emph{isometries},
things which are certainly not
true for the monoid of all bounded linear operators on Hilbert space. It seems likely that the algebra of $\Perm$ has many other unexpected special properties.

\section{Future work} \label{sec:future}

Our counterexample to continuity shows that
the statement 
\begin{equation} \label{eq:continuity conjecture}
    \mbox{If } \sigma_N \to P, \tau_N \to Q, \mbox{ then } \sigma_N \tau_N \to PQ
\end{equation}
is false in general. 
However, our counterexample exploited a close relationship between $\sigma_N$ and $\tau_N$
($\sigma_N = \tau_N^{-1}$).
Note that if we assume weak convergence $P_n \to P$
and $Q_n \to Q$, we can identify a quantity which is the obstruction to weak convergence of the compositions
as follows.
For any $f, g \in L^2(I, \mathcal{L})$, we have
\begin{align*}
    \pair{\widehat{P}_n \widehat{Q}_n f}{g} - \pair{\widehat{P} \widehat{Q} f}{g}
    =
    \pair{\widehat{P}_n (\widehat{Q}_n - \widehat{Q}) f}{g}
    + \pair{(\widehat{P}_n - \widehat{P})\widehat{Q} f}{g}
    = \pair{(\widehat{Q}_n - \widehat{Q})f}{\widehat{P}_n^\dagger g} + o(1).
\end{align*}
If the left hand side is equal to $0$ for all $f,g$,
then we have weak convergence. Thus, the only way for weak convergence to fail is if the first term on the rightmost side does not vanish; even though the error term $(\widehat{Q}_n - \widehat{Q})f$
is asymptotically orthogonal to any fixed function, some
$\widehat{P}^\dagger_n g$ is ``following''  this error term around. Note that using weak convergence again, one can similarly show that $P_n Q_n \to PQ$
weakly if and only if $\pair{(Q_n -Q)f}{(P_n - P)^\dagger g}$
goes to $0$ for all fixed $f,g$.

Therefore, it appears that $P_n$ and $Q_n$ must ``conspire'' to prevent weak convergence of $P_n Q_n$,
and convergence must be much more common in applications.
For instance, one might conjecture
that, if $\sigma_N$ and $\tau_N$ are \emph{independent}, then \eqref{eq:continuity conjecture} holds. This cannot quite be true,
since our counterexample argument yields \emph{deterministic} sequences $\sigma_N \to U, \tau_N \to U$ with $\sigma_N \tau_N \to \Delta \ne U$,
and such sequences would technically be independent.
However, we have not constructed any independent non-deterministic counterexamples.
Hence it seems plausible to conjecture that \eqref{eq:continuity conjecture} holds if one assumes a variance lower bound or some similar nondegeneracy condition in addition to independence.
A natural special case to consider in which \eqref{eq:continuity conjecture} likely holds is the case that $\sigma_N$ and $\tau_N$ are the permutations induced by sampling $N$ independent points each from $P$ and $Q$ respectively.

To our knowledge, the only other paper
which has thus far shown a statement of the form \eqref{eq:continuity conjecture}, even in a special case, is \cite{Jin2019}, and this is done using quite different techniques; therefore, even establishing instances of \eqref{eq:continuity conjecture} in special cases could be a worthy endeavor.

Unfortunately, since a sequence of permutations converges strongly if and only if it converges to an isometry (by \cref{characterization of strong convergence}), \cref{main} may be roughly the limit of the application of strong operator convergence to permutations. (Still, many important natural permutons satisfy the hypotheses of \cref{main}, so one may get a lot of mileage out of this.)
On the other hand, we note that our main theorems do not actually require any sort of independence assumption. We simply assumed independence in our application because (1) compositions of independent permutation sequences pose a natural question, and (2) our applications to the \emph{expectation} of the limit \emph{do} require an independence assumption. Thus \cref{main} may be useful for instances of random permutations which lack independence.

Another direction is to try to describe the composition on $\Perm$ in terms of \emph{subpattern densities}.
For $n \in \N, P \in \Perm$,
let $\mathrm{perm}(P,n) \in S_n$ be the random permutation induced by the random set $\{(X_1,Y_1), ..., (X_n,Y_n)\} \subset I^2$,
where the $(X_i, Y_i) \in I^2$
are independently sampled with distribution $P$.\footnote{For a subset $S \subset I^2$, of size $n$, the \emph{permutation induced by $S$} is the unique $\sigma \in S_n$ which matches the relative order of $S$. That is,
there exists a (unique) map $\phi:S \to \{1,...,n\}$
such that for all $(x,y), (x',y') \in S$ we have
$x < x' \Leftrightarrow \phi(x,y) < \phi(x',y')$
and $y < y' \Leftrightarrow \sigma(\phi(x,y)) < \sigma(\phi(x',y'))$.}
For a random permuton $P$
and $\tau \in S_n$,
the \emph{expected subpattern density} of $\tau$ in $P$ is defined to
be $\Pbb(\mathrm{perm}(P,n) = \tau)$.
Theorem 2.5 and Proposition 2.9 of \cite{Bassino2020}
establish that
the map
\[
    \Prob(\Perm) \to 
    [0,1]^{\bigcup_{n=1}^\infty S_n},
    P \mapsto
    \big(\Pbb( \mathrm{perm}(P,n) = \tau) \big)_{n \in \N, \tau \in S_n}
\]
is an affine\footnote{Note here that the affine structure is that of $\Prob(\Perm)$, \emph{not}
that of $\Perm$. Even though $\Perm$ naturally embeds continuously into $\Prob(\Perm)$ via the map $P \mapsto \delta_P$,
beware that this map is \emph{not} affine; typically $s\delta_P + t\delta_Q \ne \delta_{sP + tQ}$.} 
homeomorphism onto its image,\footnote{Here $[0,1]^{\bigcup_{n=1}^\infty S_n}$ is endowed with the product topology.} and that the image of this map consists of ``consistent families of random permutations'' (see Definition 2.8 of \cite{Bassino2020}). That is, random permutons can equivalently
be viewed as consistent families of random permutations.
Thus, multiplication of independent random permutons induces a semigroup structure on the collection of consistent families of random permutations inside of $[0,1]^{\bigcup_{n=1}^\infty S_n}$. The extreme points of this set correspond to deterministic permutons, so the restriction of the composition operation to these extreme points
gives a semigroup isomorphic to $\Perm$.
It would be interesting to see if one can relate the composition operation on $\Perm$ (and $\Prob(\Perm)$)
to the compositions of the associated random permutations. This problem is related to the question of whether \eqref{eq:continuity conjecture} holds when $\sigma_N = \mathrm{perm}(P,N), \tau_N = \mathrm{perm}(Q,N)$.

Another natural question to attack is whether for independent (biased) Brownian separable permutons
$P_1, P_2,...,$
we have that 
\[
    P_1 \cdots P_k \to U
\]
as $k \to \infty$ almost surely, or in probability.
This would be consistent with our result
that $\E[P_1 \cdots P_k] \to U$. Such a result would say that, from the point of view of subpattern densities, a large product of independent separable permutations on a large set is nearly indistinguishable from a permutation sampled uniformly from a large set. That is, the structure imposed
by avoiding certain subpatterns
fades away as one takes independent compositions.

It is instructive to contrast this situation with that of e.g. $231$-avoiding permutations.
The condition of being $231$-avoiding
enforces strong structural constraints
on the permutation, which is witnessed
by the convergence to the \emph{diagonal}
permuton $\Delta$.
Moreover, from \cref{corollary diagonal},
we know that any power or finite product of uniformly chosen $231$-avoiding permutations will also converge to $\Delta$. Thus, in the situation of $231$-avoiding permutations,
finite compositions actually \emph{preserve} the structure of the limit (at least at the level of permutons), even though the compositions themselves will not necessarily
be $231$-avoiding.
This suggests the general problem of determining which forbidden patterns produce
structures robust to (independent) composition and which do not. 

In another direction, since permutons appear across so many domains, it may be fruitful to try to further apply theory across these domains. For instance, a Lebesgue measure-preserving map $F$ on the interval\footnote{One can also take an arbitrary compact metrizable space $X$ and probability measure $\mu$. Most of the theory developed here gives identical results for ``permutons on $X$'', that is, measures on $X \times X$ with marginals equal to $\mu$. In fact, if $(X,\mu)$
is \emph{measure-theoretically} isomorphic to $(I,\mathcal{L})$ (a ``standard probability space''),
that isomorphism induces an affine homeomorphism and semigroup isomorphism between ``permutons on $X$'' and $\Perm$.
Thus, the space of permutons does not remember the topology of $X$, only the measure-theoretic properties of $(X,\mu)$.}
defines a permuton $P = \gr(F)$, and
$F$ is conjugate to another map $G$ on $I$
if and only if $\gr(G) = QPQ^\dagger$
for some invertible $Q \in \Perm$. In other words,
$\Perm$ gives a natural setting for studying probability measure-preserving dynamical systems up to measurable conjugacy, and indeed to some extent this approach has already been used to great effect (see e.g. \cite{Glasner2003}).
Since in the combinatorial context, permutons have such a close connection with subpattern densities, it may be interesting to consider whether there are any \emph{measurable conjugacy} invariants for measure-preserving systems which can be phrased in terms of subpattern densities of the associated permutons.

\bibliographystyle{plain}
\bibliography{refs}

\appendix
\section{Some facts about Markov kernels}
Here, to keep the paper mostly self-contained, we prove some basic results about stationary Markov kernels which are already known,
but may be slightly difficult to find a proof of in the literature.

Firstly, we have that composition of \emph{stationary} Markov kernels is well-defined when we quotient by the equivalence relation of almost-sure equality.
That is:

\begin{prop}
   Let $\mu, \mu', \nu, \nu': I \to \mathrm{Prob}(I)$ all be Markov kernels for which the uniform measure $\mathcal{L}$ on $I$ is stationary.
   Further suppose that
   \[
      \mathcal{L} \{ x : \mu_x = \mu'_x \} = 1 = \mathcal{L} \{ x : \nu_x = \nu'_x \}.
   \]
   Then 
   \[
      \mathcal{L} \{ x : (\nu \circ \mu)_x = (\nu' \circ \mu')_x \} = 1.
   \]
\end{prop}
\begin{proof}
Firstly, whenever $\mu_x = \mu'_x$, we have 
\[
(\nu \circ \mu)_x = \int_{y \in I} \nu_y d\mu_x(y) = \int_{y \in I} \nu_y d\mu'_x(y) = (\nu \circ \mu')_x.
\]
So $(\nu \circ \mu)_x = (\nu \circ \mu')_x$ for $\mathcal{L}$-almost every $x$.
Now, given an arbitrary Borel $B \subset I$, we want to show that $(\nu \circ \mu')_x(B) = (\nu' \circ \mu')_x(B)$
for $\mathcal{L}$-almost every $x$. For this, first note that
\begin{align*}
   0 &\le
   |(\nu \circ \mu')_x(B) - (\nu' \circ \mu')_x(B)| \\
   &= \left|\int_{y \in I} \nu_y(B) d\mu'_x(y) - \int_{y \in I} \nu'_y(B) d\mu'_x(y)\right| \\
   &= \left|\int_{y \in I} (\nu_y(B) - \nu'_y(B)) d\mu'_x(y)\right| \\
   &\le \int_{y \in I} |\nu_y(B) - \nu'_y(B)| d\mu'_x(y).
\end{align*}
So, integrating the nonnegative function $x \mapsto |(\nu \circ \mu')_x(B) - (\nu' \circ \mu')_x(B)|$ with respect to Lebesgue measure gives
\begin{align*}
   \int_0^1 |(\nu \circ \mu')_x(B) - (\nu' \circ \mu')_x(B)| dx
   &\le \int_{x=0}^{x=1} \int_{y \in I} |\nu_y(B) - \nu'_y(B)| d\mu'_x(y) dx \\
   &= \int_{y=0}^{y=1} |\nu_y(B) - \nu'_y(B)| dy = 0,
\end{align*}
where in the last line we have used that by assumption $ \mu'_* \mathcal{L} = \mathcal{L}$, and
that $\nu_y = \nu'_y$ $\mathcal{L}$-almost surely.
Thus, we have $|(\nu \circ \mu')_x(B) - (\nu' \circ \mu')_x(B)| = 0$ for $\mathcal{L}$-a.e. $x$.
Letting $B$ vary over a countable generator for the $\sigma$-algebra of $I$ (e.g. the set of all closed intervals with rational endpoints),
we have that 
\[
   (\nu \circ \mu)_x = (\nu \circ \mu')_x = (\nu' \circ \mu')_x
\]
for $\mathcal{L}$-a.e. $x$, as desired.
\end{proof}

\begin{prop}
   There is a bijection
   \[
      \{ \mathrm{permutons} \} \leftrightarrow \left\{ \mbox{Markov kernels on } I \mbox{ for which uniform measure is stationary}\right\}/\sim,
   \]
   where the equivalence relation $\sim$ identifies such Markov kernels which are equal on a set of uniform measure 1.
\end{prop}
\begin{proof}
   Given a permuton $P \in \mathrm{Prob}(I \times I)$ we define a Markov kernel $p$ as follows. Consider the projection onto the first coordinate
   $\pi_1: I \times I \to I$, and recall that by definition $(\pi_1)_* P = \mathcal{L}$.
   Then disintegration of measure (\cref{disintegration}) gives us a Markov kernel $I \to \mathrm{Prob}(I), x \mapsto p_x$ 
   such that $P = \int_{x \in I} (\delta_x \otimes p_x) dx$, and $p$ is unique up to $\mathcal{L}$-a.e. equality.
   We claim that $\mathcal{L}$ is a stationary measure for $p$. To see this, note that
   \begin{align*}
      p \circ \mathcal{L} = \int_0^1 p_x dx 
      = \int_0^1 \pi_2^* (\delta_x \otimes p_x) dx 
      = \pi_2^* \int_0^1 (\delta_x \otimes p_x) dx 
      = \pi_2^* P = \mathcal{L},
   \end{align*}
   where in the third equality we used linearity of pushforward, and in the last we used that $P$ is a permuton.
   
   Thus we have constructed a well-defined map from permutons to stationary Markov kernels up to $\mathcal{L}$-a.e. equality.
   We construct the inverse map as follows.
   Given a Markov kernel $\mu: I \to \mathrm{Prob}(I)$ for which uniform measure is stationary,
   define a probability measure $P_{\mu} \in \mathrm{Prob}(I)$ by
   \[
      P_{\mu} = \int_0^1 (\delta_x \otimes \mu_x) dx.
   \]
   This is a probability measure on $I \times I$, as an average of products of probability measures on $I$.
   Using linearity of pushforward, we can easily check that $P_{\mu}$ is a permuton:
   \begin{align*}
      (\pi_1)_* P_{\mu} = (\pi_1)_* \int_0^1 (\delta_x \otimes \mu_x) dx 
      = \int_0^1 (\pi_1)_* (\delta_x \otimes \mu_x) dx 
      = \int_0^1 \delta_x dx = \mathcal{L},
   \end{align*}
   and 
    \begin{align*}
      (\pi_2)_* P_{\mu} = (\pi_2)_* \int_0^1 (\delta_x \otimes \mu_x) dx 
      = \int_0^1 (\pi_2)_* (\delta_x \otimes \mu_x) dx 
      = \int_0^1 \mu_x dx = \mathcal{L},
   \end{align*}
   where in the last equality we have used that $\mathcal{L}$ is a stationary measure for $\mu$.
   
   Note also that it follows immediately from our formula for $P_{\mu}$
   that if $\mu \sim \mu'$, then $P_{\mu} = P_{\mu'}$.
   Thus we have constructed a well-defined map in the opposite direction.
   
   Now we show that these two maps $P \mapsto p$ and $\mu \mapsto P_\mu$ are inverse to each other.
   
   First, given a Markov kernel $\mu$ (with uniform stationary measure), let us show that $p_{\mu} \sim \mu$. This follows immediately from $\mathcal{L}$-a.e. uniqueness of disintegration of measure, since
   \[
       \int_{x \in I} (\delta_x \otimes \mu_x)dx
       = P_\mu
       = \int_{x \in I} (\delta_x \otimes p_\mu)dx.
   \]

   Next, given a permuton $P$, we must show that $P_{p} = P$. But again, by construction,
   $p_x$ gives a disintegration for $P$, so we have
   \[
      P_{p} = \int_0^1 (\delta_x \otimes p) dx = P,
   \]
   as desired.
\end{proof}

\section{Proof of \cref{basic operator facts}}
Here we prove \cref{basic operator facts}, which we divide up into several propositions.
Other proofs can be found in e.g. Section 13 of \cite{Eisner2015}.

\begin{prop}
   Fix any permuton $P$.
   For any bounded measurable function $f:I \to \C$ and any $p \in [1,\infty]$, we have $\| \widehat{P}f \|_p \le \|f\|_p$,
   where $\widehat{P}f$
   is defined as in \eqref{eq:operator definition}
   and 
   $\| f \|_p := \left( \int_I |f(x)|^p dx \right)^{1/p}$ is the $L^p$ norm with respect to the measure $\mathcal{L}$.
   In particular, $\widehat{P}$ defines a bounded linear operator on each $L^p(I, \mathcal{L})$.
\end{prop}
\begin{proof}
   Let $p \in [1, \infty)$ and let $f:I \to \C$ be a bounded measurable function. Then by Jensen's inequality, we have
   \begin{align*}
      \| \widehat{P}f \|_p^p &= \int_{x \in I} \left| \int_{y \in I} f(y) d(p^{\dagger})_x(y) \right|^p dx \\
      &\le \int_{x \in I} \int_{y \in I} |f(y)|^p d(p^{\dagger})_x(y) dx \\ 
      &= \int_{x \in I} |f(x)|^p dx = \| f \|_p^p,
   \end{align*}
   where in the last line we used that $\mathcal{L}$ is stationary for $p^{\dagger}$.
   Taking $p$th roots of both sides gives the desired inequality.
   
   To deduce the inequality for $p = \infty$, one can take the limit as $p \to \infty$ (or argue directly, which is straightforward).
\end{proof}

\begin{prop} \label{inner product}
   Let $f, g \in L^2(I, \mathcal{L})$. Then 
   \[
      \int_{I \times I} f(x) \overline{g(y)} dP(x,y) = \langle \widehat{P}f, g \rangle,
   \]
   where $\langle \cdot, \cdot \rangle$, denotes the inner product on $L^2(I, \mathcal{L})$.
\end{prop}
\begin{proof}
   We compute
   \begin{align*}
      \langle \widehat{P}f, g \rangle &=
      \int_{x \in I} (\widehat{P}f)(x) \overline{g(x)} dx \\
      &= \int_{x \in I} \int_{y \in I} f(y) d(p^{\dagger})_x(y) \overline{g(x)} dx \\
      &= \int_{x \in I} \int_{y \in I} f(y) \overline{g(x)} d(p^{\dagger})_x(y)  dx \\
      &= \int_{y \in I} \int_{x \in I} f(y) \overline{g(x)} dp_y(x) dy \\
      &= \int_{x \in I} \int_{y \in I} f(x) \overline{g(y)} dp_x(y) dx \\
      &= \int_{I \times I} f(x) \overline{g(y)} dP(x,y).
   \end{align*}
\end{proof}

\begin{prop}
   Let $P \in \mathrm{Perm}$ and let $P^{\dagger}$ be its adjoint permuton. Then the operator $\widehat{P^\dagger}$ associated to $P^{\dagger}$
   is the adjoint of the operator $\widehat{P}$ associated to $P$.
\end{prop}
\begin{proof}
   Let $f,g \in L^2(I, \mathcal{L})$. Then we have
   \begin{align*}
      \langle \widehat{P}f, g \rangle &=
      \int_{I \times I} f(x) \overline{g(y)} dP(x,y) \\
      &= \int_{I \times I} f(y) \overline{g(x)}   dP^\dagger(x,y) \\
      &= \overline{\int_{I \times I} g(x) \overline{f(y)}  dP^\dagger(x,y)} \\
      &= \overline{\langle \widehat{P^{\dagger}} g, f \rangle} \\
      &= \langle f, \widehat{P^\dagger} g \rangle.
   \end{align*}
\end{proof}

\begin{prop}
   For any $p \in [1, \infty]$, the map $\mathrm{Perm} \to \mathcal{B}(L^p(I, \mathcal{L}))$ is injective.
\end{prop}
\begin{proof}
   Fix $p \in [1, \infty]$.
   Let $P, Q \in \mathrm{Perm}$ and suppose that the induced operators $\widehat{P}$ and $\widehat{Q}$ on $L^p(I, \mathcal{L})$ are the same.
   We want to show that $P = Q$ as measures on $I \times I$.
   
   It suffices to prove that for any $A, B \subset I$ Borel, $P(A \times B) = Q(A \times B)$
   (since these form a $\pi$-system which generates the Borel $\sigma$-algebra of $I \times I$).
   But since $\ind_A, \ind_B \in L^p(I, \mathcal{L}) \cap L^2(I, \mathcal{L})$,
   $P \ind_A = Q \ind_A$ and we have
   \begin{align*}
      P( A \times B) &=
      \int_{I \times I} \ind_A(x) \ind_B(y) dP(x,y) \\
      &= \langle P \ind_A, \ind_B \rangle \\
      &= \langle Q \ind_A, \ind_B \rangle \\
      &= \int_{I \times I} \ind_A(x) \ind_B(y) dQ(x,y) \\
      &= Q(A \times B),
   \end{align*}
   as desired.
\end{proof}

To show that $P \mapsto \widehat{P}$
is a monoid homomorphism, we actually first show that this map takes the action of the Markov kernel $p$ on (absolutely continuous) \emph{measures}
to the action of $\widehat{P}$ on their Radon-Nikodym derivatives (densities) with respect to $\mathcal{L}$.

\begin{prop} \label{takes action to action}
   Suppose $f \in L^1(I, \mathcal{L}), f \ge 0, \int_{x \in I} f(x) dx = 1$. Denote by $f(x)dx$ the probability measure $A \mapsto \int_A f(x) dx$.
   Then $(\widehat{P}f)(x)dx = p \circ [ f(x)dx ]$, where composition is in the sense of composition of Markov kernels with measures.
\end{prop}
\begin{proof}
   Fix $P \in \mathrm{Perm}$.
   Fix $f \in L^1(I)$.
   Since by definition,
   \[
      p \circ [f(x) dx] = \int_{x \in I} p_x f(x) dx,
   \]
   for any $g \in C(I)$, we have
   \begin{align*}
      \int_{y \in I} g(y) d\left[ p \circ [f(x) dx] \right](y) &=
      \int_{x \in I} \int_{y \in I} g(y) dp_x(y) f(x) dx \\
      &=\int_{y \in I} \int_{x \in I} g(y) f(x) d(p^{\dagger})_y(x) dy \\
      &=\int_{y \in I} \int_{x \in I} g(y) \left(\int_{x \in I} f(x) d(p^{\dagger})_y(x)\right) dy \\
      &=\int_{y \in I} g(y) (\widehat{P}f)(y) dy.
   \end{align*}
   Therefore $p \circ [f(x) dx] = (\widehat{P}f)(x)dx$, as desired.
\end{proof}

\begin{prop}
   For any $p \in [1, \infty]$, the map $\mathrm{Perm} \to \mathcal{B}(L^p(I, \mathcal{L})), P \mapsto \widehat{P}$ is a semigroup homomorphism.
   That is, if $P, Q \in \mathrm{Perm}$, then 
   $\widehat{PQ} = \widehat{P} \widehat{Q}$.
\end{prop}
\begin{proof}
   Fix any $f \in L^p(I, \mathcal{L})$ and note that $f \in L^1(I, \mathcal{L})$ as well. 
   First assume that $f \ge 0$ and $\int_{x \in I} f(x) dx = 1$.
   Then, by \cref{takes action to action} we have
   \begin{align*}
      [\widehat{PQ}f(x)]dx &=
      (p \circ q) \circ [f(x) dx] \\
      &= p \circ \left( q \circ [f(x) dx] \right) \\
      &= p \circ [ \widehat{Q}f(x) dx ] \\
      &= [\widehat{P}(\widehat{Q}f)](x) dx,
   \end{align*}
   where in the second line we have used associativity of composition of Markov kernels (and the fact that we can consider a measure on $I$
   to be a Markov kernel from the one-point space to $I$).
   Considering positive and negative parts
   and using linearity then shows that $(\widehat{PQ})f = \widehat{P}\widehat{Q}f$ for an arbitrary $f \in L^p(I, \mathcal{L})$, and so $\widehat{PQ} = \widehat{P}\widehat{Q}$ as operators on $L^p(I, \mathcal{L})$, as desired.
\end{proof}

\begin{prop}
   For any $p \in [1,\infty]$, the map $\mathrm{Perm} \to \mathcal{B}(I, \mathcal{L})$
   is affine. That is, for any permutons $P$ and $Q$ and any $s, t \in (0,1)$ with $s+t = 1$,
   $\widehat{sP + tQ} = s\widehat{P} + t\widehat{Q}$.
\end{prop}
\begin{proof}
   The computation is left to the reader; it is a quick consequence of the fact that $(sP + tQ)^\dagger = sP^\dagger + tQ^\dagger$
   (with all these interpreted as permutons).
\end{proof}

\section{Proof of \cref{characterization of weak convergence}}
\begin{proof}[Proof of \cref{characterization of weak convergence}]
   It is immediate that $(2) \Rightarrow (3)$.
   
   To see that $(3) \Rightarrow (1)$,
   first fix $p \in [1, \infty]$ and suppose that $\widehat{P}_n \to \widehat{P}$ in the weak operator topology on $L^p(I, \mathcal{L})$.
   First we show that if $f, g \in C(I)$, then 
   \[
      \int_{I \times I} f(x) g(y) dP_n(x,y) \to \int_{I \times I} f(x) g(y) dP(x,y).
   \]
   To see this, first note that $( h \mapsto \int_{y \in I} h(y) g(y) dy ) \in L^p(I, \mathcal{L})^*$
   (since $g \in C(I) \subset L^\infty(I, \mathcal{L})$);
   denote this functional by $\phi_g$.
   Moreover, $f \in C(I) \subset L^p(I, \mathcal{L})$.
   So by the definition of weak operator convergence
   and item (2) of \cref{basic operator facts}, we have
   \begin{align*}
      \int_{I \times I} f(x) g(y) dP_n(x,y) &= \int_{y \in I} (\widehat{P}_n f)(y) g(y) dy 
      = \phi_g( \widehat{P}_n f) 
      \to \phi_g(\widehat{P} f) 
      %&= \int_{y \in I} (\widehat{P} f)(y) g(y) dy \\
      = \int_{I \times I} f(x) g(y) dP(x,y),
   \end{align*}
   as claimed.
   
   Next, consider the set $\mathcal{A}$ of all finite sums of such functions, that is, the collection of
   all functions of the form $f_1(x) g_1(y) + \cdots + f_N(x) g_N(y) \in C(I \times I)$,
   where each $f_i, g_i \in C(I)$. 
   By linearity of integration, it quickly follows from the above that for any $h \in \mathcal{A}$ we have $\int_{I \times I} h dP_n \to \int_{I \times I} h dP$.
   It is also straightforward to check that $\mathcal{A}$ is a subalgebra
   of $C(I \times I)$ which separates points of $I \times I$
   and is closed under complex conjugation.
   Therefore, by the Stone-Weierstrass theorem, $\mathcal{A}$ is uniformly dense
   in $C(I \times I)$,
   and the rest of the argument follows by approximation.
   
%   Therefore, the remainder of the argument follows by approximation:
%   given $h \in C(I \times I)$, for any fixed $\epsilon > 0$ we can find $h_\epsilon \in \mathcal{A}$ with $\sup_{(x,y) \in I \times I} |h(x,y) - h_{\epsilon}(x,y)| \le \epsilon$.
%   We know that $\int_{I \times I} h_{\epsilon} dP_n \to \int_{I \times I} h_{\epsilon} dP$, and so
%   \begin{align*}
%      &\limsup_{n \to \infty} \left| \int_{I \times I} h dP_n - \int_{I \times I} h dP \right| \\
%      &\le \limsup_{n \to \infty} \int_{I \times I} |h - h_{\epsilon}| dP_n 
%      + \int_{I \times I} |h - h_{\epsilon}| dP
%      + \left| \int_{I \times I} h_{\epsilon} dP_n - \int_{I \times I} h_{\epsilon} dP \right| \\
%      &\le 2 \epsilon.
%   \end{align*}
%   Taking $\epsilon \to 0$ shows that $\int_{I \times I} h dP_n \to \int_{I \times I} h dP$
%   for all $h \in C(I \times I)$, that is, $P_n \to P$ weak-* in $C(I \times I)^*$, as desired.
   
   Now let us show that $(1) \Rightarrow (2)$.
   Suppose that $P_n \to P$ weak-* in $C(I \times I)^*$.
   Fix $p \in (1, \infty)$, and let $f \in L^p(I, \mathcal{L})$, $\phi \in L^p(I, \mathcal{L})^*$.
   Since $p \in (1, \infty)$, $\phi$ is necessarily of the form
   \[
      \phi(h) = \int_{y \in I} h(y) \overline{g(y)} dy
   \]
   for some $g \in L^q(I, \mathcal{L})$, where $q \in (1, \infty)$ satisfies $\frac{1}{p} + \frac{1}{q} = 1$.
   Now, given $\epsilon > 0$, choose $f_{\epsilon}, g_{\epsilon} \in C(I)$ such that
   \[
      \| f - f_{\epsilon}\|_p, \|g - g_{\epsilon}\|_q \le \epsilon.
   \]
   (Here we have used that $p, q < \infty$.)
   
   In the following, we use the notation $\langle f, g \rangle$ for $\int_{y \in I} f(y) \overline{g(y)} dy$
   for $f \in L^p(I, \mathcal{L}), g \in L^q(I, \mathcal{L})$ (even if $p, q \ne 2$).
   We have by H\"{o}lder's inequality that
   \begin{align*}
      \left| \phi( \widehat{P}_n f) - \phi( \widehat{P} f) \right|
      &\le
      |\langle \widehat{P}_n f_{\epsilon}, g_{\epsilon} \rangle - \langle \widehat{P} f_{\epsilon}, g_{\epsilon}  \rangle |\\
      &+ \left|\langle (\widehat{P}_n - 
      \widehat{P}) (f_{\epsilon} - f), g_{\epsilon} \rangle \right| + \left| \langle (\widehat{P}_n - \widehat{P})f, g - g_{\epsilon} \rangle \right| \\
      &\le \left| \int_{I \times I} f_{\epsilon}(x) g_{\epsilon}(y) dP_n(x,y) - \int_{I \times I} f_{\epsilon}(x) g_{\epsilon}(y) dP(x,y) \right| \\
      &+ \| (\widehat{P}_n - \widehat{P}) (f_{\epsilon} - f ) \|_p \| g_{\epsilon} \|_q + \| (\widehat{P}_n - \widehat{P}) f \|_p \|g - g_{\epsilon}\|_q.
   \end{align*}
   Using the fact that $\| P_n \|, \|P\| \le 1$ (where these norms are the operator norm on $\mathcal{B}(L^p(I, \mathcal{L}))$),
   and using the approximation properties of $f_{\epsilon}, g_{\epsilon}$, we have
   \[
      \| (\widehat{P}_n - \widehat{P}) (f_{\epsilon} - f ) \|_p \| g_{\epsilon} \|_q + \| (\widehat{P}_n - \widehat{P}) f \|_p \|g - g_{\epsilon}\|_q
      \le 2 \epsilon (\|g\|_q + \epsilon) + 2 \|f\|_p \epsilon.
   \]
   Using this and the fact that $f_{\epsilon}(x)g_{\epsilon}(y) \in C(I \times I)$ and $P_n \to P$ weak-*, taking $\limsup_{n \to \infty}$ above gives
   \[
      \limsup_{n \to \infty} |\phi(\widehat{P}_n f) - \phi(\widehat{P} f)| \le 2\epsilon [ \|g\|_q + \epsilon + \|f\|_p].
   \]
   Taking $\epsilon \to 0$ then shows that $\phi(\widehat{P}_n f) \to \phi(\widehat{P} f)$. So we have shown that
   $\widehat{P}_n \to \widehat{P}$ in the weak operator topology on $\mathcal{B}(L^p(I, \mathcal{L})$.

    (3) $\Leftrightarrow$ (4):
    Note that for any $P \in \Perm$ we have $P(A \times B) = \pair{\widehat{P}\ind_A}{\ind_B}$ by item (2)
    of \cref{basic operator facts}.
    Hence, weak operator convergence of $\widehat{P}_N$ to $\widehat{P}$ implies $P_n(A \times B) \to P(A \times B)$
    for all Borel $A,B \subset I$. The other direction is a standard approximation argument.
\end{proof}

\end{document}